\documentclass
[
    a4paper,
    DIV=11,
    abstracton
]
{scrartcl}

\usepackage
{
    graphicx,
    amssymb,
    amsmath,
    amsthm,
    dsfont, 
    xcolor,
    mathtools,
    kbordermatrix,
    authblk,
    enumitem,
}

\usepackage[pdffitwindow=false,
            plainpages=false,
            pdfpagelabels=true,
            pdfpagemode=UseOutlines,
            pdfpagelayout=SinglePage,
            bookmarks=false,
            colorlinks=true,
            hyperfootnotes=false,
            linkcolor=blue,
            urlcolor=blue!30!black,
            citecolor=green!50!black]{hyperref}

\usepackage[bf,normal]{caption}

\newcommand{\subfiguretitle}[1]{{\scriptsize{#1}} \\}
\newcommand{\R}{\mathbb{R}}                                     
\newcommand{\pd}[2]{\frac{\partial#1}{\partial#2}}              
\newcommand{\innerprod}[2]{\left\langle #1,\, #2 \right\rangle} 
\newcommand{\ts}{\hspace*{0.1em}}                               
\newcommand{\relmiddle}[1]{\mathrel{}\middle#1\mathrel{}}       
\providecommand{\norm}[1]{\left\lVert #1 \right\rVert}          

\newcommand\xqed[1]{\leavevmode\unskip\penalty9999 \hbox{}\nobreak\hfill \quad\hbox{#1}}
\newcommand{\exampleSymbol}{\xqed{$\triangle$}}

\let\Im\relax\DeclareMathOperator{\Im}{Im}

\newtheorem{theorem}{Theorem}[section]

\newtheorem{proposition}[theorem]{Proposition}

\theoremstyle{definition}
\newtheorem{example}[theorem]{Example}
\newtheorem{remark}[theorem]{Remark}

\makeatletter
\renewcommand*\env@matrix[1][*\c@MaxMatrixCols c]{%
  \hskip -\arraycolsep
  \let\@ifnextchar\new@ifnextchar
  \array{#1}}
\makeatother

\mathtoolsset{centercolon} 

\title{Data-driven approximation of the \\ Koopman generator: Model reduction, system identification, and control}
\author[1]{Stefan Klus}
\author[2,3]{Feliks N\"uske}
\author[3]{Sebastian Peitz}
\author[4]{\\Jan-Hendrik Niemann}
\author[2]{Cecilia Clementi}
\author[1,4]{Christof Sch\"utte}
\affil[1]{Department of Mathematics and Computer Science, Freie Universit\"at Berlin, Germany}
\affil[2]{Center for Theoretical Biological Physics and Department of Chemistry, Rice University, USA}
\affil[3]{Department of Mathematics, Paderborn University, Germany}
\affil[4]{Zuse Institute Berlin, Germany}

\date{}

\begin{document}
\maketitle

\begin{abstract}
We derive a data-driven method for the approximation of the Koopman generator called gEDMD, which can be regarded as a straightforward extension of EDMD (\emph{extended dynamic mode decomposition}). This approach is applicable to deterministic and stochastic dynamical systems. It can be used for computing eigenvalues, eigenfunctions, and modes of the generator and for system identification. In addition to learning the governing equations of deterministic systems, which then reduces to SINDy (\emph{sparse identification of nonlinear dynamics}), it is possible to identify the drift and diffusion terms of stochastic differential equations from data. Moreover, we apply gEDMD to derive coarse-grained models of high-dimensional systems, and also to determine efficient model predictive control strategies. We highlight relationships with other methods and demonstrate the efficacy of the proposed methods using several guiding examples and prototypical molecular dynamics problems.
\end{abstract}

\section{Introduction}

Data-driven approaches for the analysis of complex dynamical systems---be it methods to approximate transfer operators for computing metastable or coherent sets, methods to learn physical laws, or methods for optimization and control---have been steadily gaining popularity over the last years. Algorithms such as DMD~\cite{Schmid10, TRLBK14}, EDMD~\cite{WRK15, KKS16}, SINDy~\cite{BrPrKu16}, and their various kernel- \cite{WRK15, SP15, KSM19}, tensor- \cite{KGPS18, GKES19, CSBR19}, or neural network-based \cite{LDBK17, LKB17, MPWN18} extensions and generalizations have been successfully applied to a plethora of different problems, including molecular and fluid dynamics, meteorology, finance, as well as mechanical and electrical engineering. An overview of different applications can be found, e.g., in \cite{KBBP16}. Similar methods, developed mainly for reversible molecular dynamics problems, have been proposed in \cite{NKPMN14}. Most of the aforementioned techniques turn out to be strongly related, with the unifying concept being Koopman operator theory \cite{Ko31, LaMa94, BMM12}. In what follows, we will focus mainly on the generator of the Koopman operator and its properties and applications.

SINDy~\cite{BrPrKu16} constitutes a milestone for data-driven discovery of dynamical systems. Because of the close relationship between the vector field of a deterministic dynamical system and its Koopman generator, SINDy is a special case of the framework we will introduce in this study. In \cite{Kaiser17, Kaiser18}, an extension of SINDy to determine eigenfunctions of the Koopman generator was presented. The discovered eigenfunctions are then used for control, resulting in the so-called KRONIC framework. Another extension of SINDy was derived in \cite{BNC18}, allowing for the identification of parameters of a stochastic system using Kramers--Moyal formulae.

A different avenue towards system identification was taken in~\cite{MauGon16, MauGon17}. Here, the Koopman operator is first approximated with the aid of EDMD, and then its generator is determined using the matrix logarithm. Subsequently, the right-hand side of the differential equation is extracted from the matrix representation of the generator. The relationship between the Koopman operator and its generator was also exploited in \cite{RiTK17} for parameter estimation of stochastic differential equations.

A method for computing eigenfunctions of the Koopman generator was proposed in~\cite{Gia19}, where the diffusion maps algorithm is used to set up a Galerkin-projected eigenvalue problem with orthogonal basis elements. Two efficient methods for computing the generator of the adjoint Perron--Frobenius operator based on Ulam's method and spectral collocation were presented in \cite{FrJuKo13}. Provided that a model of the system dynamics is available, the computation of trajectories can be replaced by evaluations of the right-hand side of the system, which is often orders of magnitude faster.

The purpose of this study is to present a general framework to compute a matrix approximation of the Koopman generator, both for deterministic and stochastic systems, and to explore a range of applications. The main contributions of this work are:
\begin{enumerate}[wide, itemindent=\parindent, itemsep=0ex, topsep=0.5ex]
\item We reformulate standard EDMD in such a way that it can be used to approximate the generator of the Koopman operator---as well as its eigenvalues, eigenfunctions, and modes---from data without resorting to trajectory integration. Exploiting duality, this can be extended naturally to the generator of the Perron--Frobenius operator.
\item We illustrate that the governing equations of deterministic as well as stochastic dynamical systems can be obtained from empirical estimates of the generator. Furthermore, we highlight relationships with related system identification techniques such as the Koopman lifting approach \cite{MauGon16}, SINDy \cite{BrPrKu16}, and KRONIC \cite{Kaiser17}, which focus mainly on identifying ordinary differential equations.
\item Lastly, we explore two powerful applications of the approximated Koopman generator. We show that gEDMD can be used to identify coarse-grained models based on data of the full system, which is a highly relevant topic across different research fields, like molecular dynamics simulations for instance. Moreover, we apply the Koopman generator to control dynamical systems, providing flexible and efficient model predictive control strategies.
\end{enumerate}

The efficacy of the resulting methods will be demonstrated with the aid of guiding examples and illustrative benchmark problems.

The remainder of this paper is structured as follows: In Section~\ref{sec:Koopman operator}, we introduce the Koopman operator and its generator for different kinds of dynamical systems. We then derive an extension of EDMD for the approximation of the Koopman generator, named gEDMD, in Section~\ref{sec:gEDMD}. Furthermore, relationships with other methods are described. Section~\ref{sec:Further applications} explores additional applications of the proposed methods, namely coarse-graining and the application to control problems. Open questions and future work are discussed in Section~\ref{sec:Conclusion}.

\section{The Koopman operator and its generator}
\label{sec:Koopman operator}

In what follows, let $ \mathbb{X} $ be the state space, e.g., $ \mathbb{X} \subset \R^d $, and $ f \in L^{\infty}(\mathbb{X}) $ a real-valued observable of the system.

\subsection{Deterministic dynamical systems}

Given an ordinary differential equation of the form $ \dot{x} = b(x) $, where $ b \colon \R^d \to \R^d $, the so-called \emph{Koopman semigroup} of operators $ \{\ts \mathcal{K}^t \ts\} $ is defined as
\begin{equation*}
    (\mathcal{K}^t f)(x) = f(\Phi^t(x)),
\end{equation*}
where $ \Phi^t $ is the flow map, see \cite{LaMa94, BMM12, KKS16}. That is, if $ x(t) $ is a solution of the initial value problem with initial condition $ x(0) = x_0 $, then $ \Phi^t(x_0) = x(t) $. The infinitesimal generator $ \mathcal{L} $ of the semigroup, defined as
\begin{equation*}
    \mathcal{L} f = \lim_{t \rightarrow 0} \frac{1}{t} \left(\mathcal{K}^t f - f \right),
\end{equation*}
is given by
\begin{equation*}
    \mathcal{L} f = \frac{\mathrm{d}}{\mathrm{d}t} f
                  = b \cdot \nabla_x f
                  = \sum_{i=1}^d b_i \ts \pd{f}{x_i},
\end{equation*}
see, e.g., \cite{LaMa94}. Thus, if $ f $ is continuously differentiable, then $ u(t, x) = \mathcal{K}^t f(x) $ satisfies the first-order partial differential equation $ \pd{u}{t} = \mathcal{L} u $. The adjoint operator $ \mathcal{L}^* $, i.e., the generator of the Perron--Frobenius operator, is given by
\begin{equation*}
    \mathcal{L}^* f = -\sum_{i=1}^d \pd{(b_i \ts f)}{x_i}.
\end{equation*}

\begin{example} \label{ex:Simple example}
Throughout the paper, we will use the simple system
\begin{equation*}
    \begin{split}
        \dot{x}_1 &= \gamma \ts x_1, \\
        \dot{x}_2 &= \delta \ts (x_2 - x_1^2),
    \end{split}
\end{equation*}
taken from~\cite{BBPK16}, as a guiding example. In addition to the trivial eigenfunction $ \varphi_1(x) = 1 $ with corresponding generator eigenvalue $ \lambda_1 = 0 $, we obtain $ \varphi_2(x) = x_1 $ and $ \varphi_3(x) = \frac{2 \gamma - \delta}{\delta} x_2 + x_1^2 $ with corresponding generator eigenvalues $ \lambda_2 = \gamma $ and $ \lambda_3 = \delta $, respectively. Moreover, products of eigenfunctions are again eigenfunctions. \exampleSymbol
\end{example}

\subsection{Non-deterministic dynamical systems}

Similarly, the definition of the Koopman operator can be generalized to stochastic differential equations
\begin{equation} \label{eq:SDE}
    \mathrm{d}X_t = b(X_t) \ts \mathrm{d}t + \sigma(X_t) \ts \mathrm{d}W_t
\end{equation}
as described, e.g., in \cite{Hol08}, resulting in
\begin{equation} \label{eq:stochastic_koopman_op}
    (\mathcal{K}^t f)(x) = \mathbb{E}[f(\Phi^t(x))].
\end{equation}
Here, $ \mathbb{E}[\,\cdot\,] $ denotes the expected value, $ b \colon \R^d \to \R^d $ is the drift term, $ \sigma \colon \R^d \to \R^{d \times s} $ the diffusion term, and $ W_t $ an $ s $-dimensional Wiener process. Given a twice continuously differentiable function $ f $, it can be shown using It\^{o}'s lemma that  the infinitesimal generator of the stochastic Koopman operator is then characterized by
\begin{equation} \label{eq:generator_SDE}
    \mathcal{L} f
    = b \cdot \nabla_x f + \frac{1}{2} a : \nabla_x^2 f
    = \sum_{i=1}^d b_i \ts \pd{f}{x_i} + \frac{1}{2} \sum_{i=1}^d \sum_{j=1}^d a_{ij} \ts \pd{^2 f}{x_i \ts \partial x_j},
\end{equation}
where $ a = \sigma \ts \sigma^\top $ and $ \nabla_x^2 $ denotes the Hessian. Properties of the generator associated with non-deterministic dynamical systems are studied in \cite{CMM19}. The function $ u(t, x) = \mathcal{K}^t f(x) $ satisfies the second-order partial differential equation $ \pd{u}{t} = \mathcal{L} u $, which is called the \emph{Kolmogorov backward equation} \cite{Met07}. The adjoint operator in this case is
\begin{equation*}
    \mathcal{L}^* f = -\sum_{i=1}^d \pd{(b_i \ts f)}{x_i}  + \frac{1}{2} \sum_{i=1}^d \sum_{j=1}^d \pd{^2 (a_{ij} \ts f)}{x_i \ts \partial x_j}
\end{equation*}
so that $ \pd{u}{t} = \mathcal{L}^* u $ becomes the \emph{Fokker--Planck equation} or \emph{Kolmogorov forward equation}~\cite{LaMa94}.

If $\mu$ is a stationary measure for the process $X_t$, the Koopman operator can be extended from $L^\infty_\mu(\mathbb{X})$ to the Hilbert space $L^2_\mu(\mathbb{X})$ with inner product $\innerprod{f}{g}_\mu = \int_\mathbb{X} f(x) \ts g(x) \ts \mathrm{d}\mu(x)$ \cite{BAXTER1995}. We will frequently consider this situation in what follows. An important class of stochastic differential equations are those which are reversible with respect to a measure $\mu$, which is necessarily a stationary measure in this case. The Koopman operator becomes self-adjoint on $L^2_\mu(\mathbb{X})$ in the reversible setting. Reversible systems can be characterized by the diffusion $\sigma$ and a scalar potential $F \colon \R^d \to \R$, from which the drift is then obtained by
\begin{equation*}
b = -\frac{1}{2} a \ts \nabla F + \frac{1}{2} \nabla \cdot a,
\end{equation*}
where the divergence in the second term is applied to each column of $a$ \cite{Pav14}. The generator of a reversible stochastic differential equation is a self-adjoint and typically unbounded operator on a suitable dense subspace of $L^2_\mu(\mathbb{X})$.

\begin{remark} \label{rem:overdamped_langevin}
For systems of the form $ \mathrm{d}X_t = -\nabla V(X_t) \ts \mathrm{d}t + \sqrt{2 \beta^{-1}} \ts \mathrm{d}W_t $, which play an important role in molecular dynamics, we obtain
\begin{equation*}
    \mathcal{L} f =  -\nabla V \cdot \nabla f + \beta^{-1} \Delta f
    \quad \text{and} \quad
    \mathcal{L}^* f =  \nabla V \cdot \nabla f + \Delta V \ts f + \beta^{-1} \Delta f.
\end{equation*}
Here, $ V $ describes the potential and $ \beta $ is the inverse temperature. The resulting dynamics are reversible with invariant measure $\mu(x) \sim \exp(-\beta \ts V(x))$. The generator $\mathcal{L}$ is self-adjoint on $L^2_\mu(\mathbb{X})$ and it can be shown that, assuming suitable growth conditions on the potential, the spectrum of $\mathcal{L}$ is discrete \cite{BAKRY2013}.
\end{remark}

\begin{example} \label{ex:Ornstein Uhlenbeck}
We will use the one-dimensional Ornstein--Uhlenbeck process, given by the stochastic differential equation
\begin{equation*}
    \mathrm{d}X_t = -\alpha \ts X_t \ts \mathrm{d}t + \sqrt{2 \beta^{-1}} \ts \mathrm{d}W_t,
\end{equation*}
which is of the above form with $ V(x) = \frac{1}{2} \ts \alpha \ts x^2 $, as a second guiding example. The parameter $ \alpha $ is the friction coefficient. The generator becomes self-adjoint in the space $ L^2(\rho) $ weighted by the invariant density
\begin{equation*}
    \rho(x) = \frac{1}{\sqrt{2 \pi \alpha^{-1} \beta^{-1}}} \exp\left(-\alpha \ts \beta \ts \frac{x^2}{2}\right)
\end{equation*}
and the eigenvalues $ \lambda_\ell $ and eigenfunctions $ \varphi_\ell $ are given by
\begin{equation*}
    \lambda_\ell = -\alpha \ts (\ell-1),
    \quad
    \varphi_\ell(x) = \frac{1}{\sqrt{(\ell-1)!}} \ts H_{\ell-1}\left(\sqrt{\alpha \beta} \ts x\right),
    \quad
    \ell = 1, 2, \dots,
\end{equation*}
where $ H_\ell $ denotes the $ \ell $th probabilists' Hermite polynomial \cite{Pav14}. That these functions are indeed eigenfunctions can be verified easily using recurrence relations for the Hermite polynomials, i.e., $ H_{\ell+1}(x) = x H_\ell(x) - H_\ell^\prime(x) $. \exampleSymbol
\end{example}

\subsection{Galerkin approximation}

Given a set of basis functions $ \{\ts \psi_i \ts\}_{i=1}^n $, where $ \psi_i \colon \R^d \to \R $, a Galerkin approximation $ \mathbf{L} $ of the generator $ \mathcal{L} $ can be obtained by computing the matrices $ A, G \in \R^{n \times n} $ with
\begin{equation} \label{eq:Galerkin_matrices}
    \begin{split}
        A_{ij} &= \innerprod{\mathcal{L} \ts \psi_i}{\psi_j}_\mu, \\
        G_{ij} &= \innerprod{\psi_i}{\psi_j}_\mu,
    \end{split}
\end{equation}
where $ \mu $ is a given measure. The matrix representation $ L $ of the projected operator $ \mathbf{L} $ is then given by $ L^\top = A \ts G^{-1} $. We define $ \psi(x) = [\psi_1(x), \dots, \psi_n(x)]^\top $. That is, for a function $ f(x) = \sum_{i=1}^n c_i \ts \psi_i(x) = c^\top \psi(x) $, it holds that $ (\mathbf{L} f)(x) = (L \ts c)^\top \psi(x) $, where $ c = [c_1, \dots, c_n]^\top \in \R^n $. It follows that an eigenvector $ \xi_\ell $ of $ L $ corresponding to the eigenvalue $ \lambda_\ell $ contains the coefficients for the eigenfunctions of $ \mathbf{L} $ since defining $ \varphi_\ell(x) = \xi_\ell^\top \psi(x) $ yields
\begin{equation*}
    (\mathbf{L} \varphi_\ell)(x) = (L \ts \xi_\ell)^\top \psi(x) = \lambda_\ell \ts \xi_\ell^\top \psi(x) = \lambda_\ell \ts \varphi_\ell(x).
\end{equation*}
In many applications, the reciprocals of the generator eigenvalues (or their approximations) are also of interest, as they can be interpreted as decay time scales of dynamical processes in the system. We will refer to them as \textit{implied time scales}
\begin{equation*}
    t_\ell := \frac{1}{\lambda_\ell}.
\end{equation*}

\begin{example} \label{ex:OrnUhl1}
For the Ornstein--Uhlenbeck process and a basis comprising monomials of order up to $ n - 1 $, i.e., $ \psi(x) = [1, x, \dots, x^{n-1}]^\top $, we can compute the matrix $ L $ analytically. Note that $ \mathcal{L} \psi_k $ is again in the subspace spanned by $ \{\ts \psi_i \ts\}_{i=1}^n $. In particular, for $ k \ge 3 $, we have
\begin{equation*}
    (\mathcal{L} \psi_k)(x) = -\alpha \ts (k-1) \ts x^{k-1} + \beta^{-1} (k-1)(k-2) \ts x^{k-3}
\end{equation*}
and the matrix $ L \in \R^{n \times n} $ is of the form
\begin{equation*}
    \renewcommand{\tabcolsep}{15cm}
    \kbordermatrix{
        & 1 & x & x^2 & x^3 & x^4 & x^5 & x^6 & \dots \\
       1 & 0 & & 2\ts\beta^{-1} & & & & & \\
       x & & -\alpha & & 6\ts\beta^{-1} & & & & \\
       x^2 & & & -2\ts\alpha & & 12\ts\beta^{-1} & & & \\
       x^3 & & & & -3\ts\alpha & & 20\ts\beta^{-1} & & \\
       x^4 & & & & & -4\ts\alpha & & 30\ts\beta^{-1} & \\[-0.9ex]
       x^5 & & & & & & -5\ts\alpha & & \ddots \\
       x^6 & & & & & & & -6\ts\alpha & \\
       \vdots & & & & & & & & \ddots
    },
\end{equation*}
where the row and column labels correspond to the respective basis functions. The eigenvalues of the generator are given by $ \lambda_\ell = -\alpha \ts (\ell-1) $, for $ \ell = 1, \dots, n $, and the resulting eigenfunctions whose coefficients are given by the eigenvectors are the (transformed) probabilists' Hermite polynomials as described above. An approach to compute Hermite polynomials by solving an eigenvalue problem, resulting in a similar matrix representation, is also described in~\cite{Aboites17}. \exampleSymbol
\end{example}

Since we in general cannot compute the required integrals analytically, the aim is to estimate them from data using, e.g., Monte Carlo integration. More details regarding different types of Galerkin approximations and other methods for the approximation of transfer operators from data can be found in \cite{KKS16, KNKWKSN18}.

\begin{remark}
Issues pertaining to non-compactness or continuous spectra of Koopman operators associated with systems of high complexity are beyond the scope of this paper. Although such cases can theoretically be handled, the numerical analysis is often challenging and typically requires regularization, which is, for instance, implicitly given by Galerkin projections \cite{Gia19}. This is discussed in detail in the aforecited work by Giannakis. Moreover, the projected generator does in general not result in a rate matrix, see \cite{SS13, SS15} for details on Galerkin discretizations of transfer operators and their properties.
\end{remark}

\section{Infinitesimal generator EDMD}
\label{sec:gEDMD}

EDMD \cite{WKR15, KKS16} was developed for the approximation of the Koopman or Perron--Frobenius operator from data. However, it can be reformulated to compute also the associated infinitesimal generators. We will call the resulting method gEDMD.

\subsection{Deterministic dynamical systems}

Let us first consider the deterministic case, which---albeit derived in another way and with different applications in mind---has already been studied in \cite{Kaiser17, Kaiser18} so that we only briefly summarize and extend these results and then generalize them to the non-deterministic setting. Detailed relationships with other methods can be found in Section~\ref{ssec:Relationships with other methods}. We now assume that we have $ m $ measurements of the states of the system, given by $ \{\ts x_l \ts\}_{l=1}^m $, and the corresponding time derivatives, given by $ \{\ts \dot{x}_l \ts\}_{l=1}^m $. The derivatives might also be estimated from data, cf.~\cite{BrPrKu16}.

\subsubsection{Generator approximation}

Similar to the Galerkin projection described above, we then choose a set of basis functions, also sometimes called \emph{dictionary}, defined by $ \{\ts \psi_i \ts\}_{i=1}^n $, and write this again in vector form as $ \psi(x) = [\psi_1(x), \dots, \psi_n(x)]^\top $. Additionally, we define
\begin{equation*}
    \dot{\psi}_k(x) = (\mathcal{L} \psi_k)(x) = \sum_{i=1}^d b_i(x) \ts \pd{\psi_k}{x_i}(x).
\end{equation*}
For all data points and basis functions, this can be written in matrix form as
\begin{equation*}
    \Psi_X =
    \begin{bmatrix}
        \psi_1(x_1) & \dots  & \psi_1(x_m) \\
        \vdots      & \ddots & \vdots      \\
        \psi_n(x_1) & \dots  & \psi_n(x_m)
    \end{bmatrix}
    \quad \text{and} \quad
    \dot{\Psi}_X =
    \begin{bmatrix}
        \dot{\psi}_1(x_1) & \dots  & \dot{\psi}_1(x_m) \\
        \vdots            & \ddots & \vdots            \\
        \dot{\psi}_n(x_1) & \dots  & \dot{\psi}_n(x_m)
    \end{bmatrix},
\end{equation*}
where $ \Psi_X, \dot{\Psi}_X \in \R^{n \times m} $. The partial derivatives of the basis functions required for $ \dot{\psi}_k(x_l) $ can be precomputed analytically.\!\footnote{Alternatively, automatic differentiation or symbolic computing toolboxes could be utilized.} Note that we additionally need $ b(x_l) $ which is simply~$ \dot{x}_l $. If the time derivatives cannot be measured directly, they can be approximated using, e.g., finite differences. We now assume there exists a matrix $ M $ such that $ \dot{\Psi}_X = M \Psi_X $. Since this equation in general cannot be satisfied exactly, we solve it in the least squares sense---analogously to the derivation of EDMD---by minimizing $ \norm{\smash{\dot{\Psi}_X - M \Psi_X}}_F $, resulting in
\begin{equation*}
    M = \dot{\Psi}_X \Psi_X^+ = \big(\dot{\Psi}_X \Psi_X^\top\big) \big(\Psi_X \Psi_X^\top\big)^+ = \widehat{A} \ts \widehat{G}^+,
\end{equation*}
with
\begin{equation*}
    \widehat{A} = \frac{1}{m} \sum_{l=1}^m \dot{\psi}(x_l) \ts \psi(x_l)^\top
    \quad \text{and} \quad
    \widehat{G} = \frac{1}{m} \sum_{l=1}^m \psi(x_l) \ts \psi(x_l)^\top.
\end{equation*}
We call this approach gEDMD. The advantage is that the generator might be sparse even when the Koopman operator for the time-$ t $ map is not.

\begin{remark}
The sparsification approach proposed for SINDy, see \cite{BrPrKu16}, can be added in the same way to gEDMD in order to minimize the number of spurious nonzero entries caused, for instance, by the numerical approximation of the time derivatives or by noisy data.
\end{remark}

The convergence to the Galerkin approximation in the infinite data limit will be shown for the non-deterministic case, the deterministic counterpart follows as a special case. The matrix $ M $ is thus an empirical estimate of $ L^\top $ and we write $ M = \widehat{L}^\top = \widehat{A} \ts \widehat{G}^+ $. Accordingly, exploiting duality, the matrix representation of the adjoint operator $ \mathcal{L}^* $, the generator of the Perron--Frobenius operator, is given by $ M^* = (\widehat{L}^*)^\top = \widehat{A}\ts^\top \widehat{G}^{+} $. A detailed derivation for standard EDMD, which can be carried over to gEDMD, can be found in \cite{KKS16}. The convergence of the standard EDMD approximation to the Koopman operator as the number of basis functions goes to infinity is discussed in \cite{KoMe18}. Whether the results can be extended to gEDMD will be studied in future work.

\begin{example} \label{ex:Simple example generator}
Let us again consider the system defined in Example~\ref{ex:Simple example} using monomials up to order $ 8 $. We set $ \gamma = -0.8 $ and $ \delta = -0.7 $ and generate $ 1000 $ uniformly distributed test points in $ [-2, 2] \times [-2, 2] $. Then gEDMD results in eigenvalues and (rescaled) eigenfunctions
\begin{equation*}
    \arraycolsep=10pt
    \def\arraystretch{1.3}
    \begin{array}{ll}
        \lambda_1 \approx 0,
            & \varphi_1(x) \approx 1, \\
        \lambda_2 \approx -0.7 = \delta,
            & \varphi_2(x) = 1.286 \ts x_2 + 1.000 \ts x_1^2 \approx \frac{2\gamma - \delta}{\delta} x_2 + x_1^2, \\
        \lambda_3 \approx -0.8 = \gamma,
            & \varphi_3(x) \approx x_1.
    \end{array}
\end{equation*}
The subsequent eigenfunctions are products of the above eigenfunctions, we obtain, for instance, $ \lambda_6 \approx -1.6 = 2 \ts \gamma $ with $ \varphi_6(x) = 1.000 \ts x_1^2 \approx \varphi_3(x)^2 $. Note that the ordering of the eigenvalues, which are typically sorted by decreasing values, and associated eigenfunctions depends on the values of $ \gamma $ and $ \delta $. \exampleSymbol
\end{example}

\subsubsection{System identification}
\label{ssec:systemidentification}

With the aid of the full-state observable $ g(x) = x $, it is possible to reconstruct the governing equations of the underlying dynamical system. Note that $ \mathbb{X} $ needs to be bounded here---and for the identification of stochastic differential equations introduced below---so that $ g $ is (component-wise) contained in $ L^{\infty}(\mathbb{X})$. Let $ \xi_\ell $ be the $ \ell $th eigenvector of $ \widehat{L} $ and $ \Xi = [\xi_1, \dots, \xi_n] $. Furthermore, assume that $ B \in \mathbb{R}^{n \times d} $ is the matrix such that $ g(x) = B^\top \ts \psi(x) $. This can be easily accomplished by adding the observables $ \{ \ts x_i \ts \}_{i=1}^d $ to the dictionary. In order to obtain the Koopman modes for the full-state observable, define $ \varphi(x) = [\varphi_1(x), \dots, \varphi_n(x)]^\top = \Xi^\top \psi(x) $. Then
\begin{align*}
    g(x) = B^\top \ts \psi(x) = B^\top \ts \Xi^{-\top} \varphi(x).
\end{align*}
The column vectors of the matrix $ V = B^\top \ts \Xi^{-\top} $ are the Koopman modes $ v_\ell $. We obtain
\begin{equation*}
    (\mathcal{L} g)(x) = b(x) \approx \sum_{\ell=1}^n \ts \lambda_\ell \ts \varphi_\ell(x) v_\ell,
\end{equation*}
where the generator is applied component-wise. This allows us to decompose a system into different frequencies. The derivation of the modes is equivalent to the standard EDMD case, see \cite{KKS16, WKR15} for more details. Instead of representing the system in terms of the eigenvalues, eigenfunctions, and modes of the generator, we can also express it directly in terms of the basis functions, i.e.,
\begin{equation*}
    (\mathcal{L} g)(x) = b(x) \approx (L B)^\top \ts \psi(x),
\end{equation*}
which is then equivalent to SINDy, see Section~\ref{ssec:Relationships with other methods}.

\begin{example} \label{ex:systemidentification}
Using the eigenvalues $ \lambda_\ell $ and corresponding eigenfunctions $ \varphi_\ell(x) $ as determined in Example~\ref{ex:Simple example generator}, we can reconstruct the dynamical system from Example~\ref{ex:Simple example}. Only the Koopman modes $ v_2 = [0, \, 0.778]^\top \approx [0, \, \frac{\delta}{2\gamma - \delta}]^\top $, $ v_3 = [1, \, 0]^\top $, and $ v_6 = [0, \, -0.778]^\top \approx [0, \, -\frac{\delta}{2\gamma - \delta}]^\top $ are required for the reconstruction, the other modes are numerically zero. That is,
\begin{equation*}
    b(x) \approx \lambda_2 \ts \varphi_2(x) \ts v_2 + \lambda_3 \ts \varphi_3(x) \ts v_3 + \lambda_6 \ts \varphi_6(x) \ts v_6 \approx
    \begin{bmatrix}
        \gamma \ts x_1 \\
        \delta \ts (x_2 - x_1^2)
    \end{bmatrix}.
\end{equation*}
Expressing the system directly in terms of the basis functions results in
the same representation, the governing equations are hence identified correctly in both cases. \exampleSymbol
\end{example}

\begin{remark} \label{rem:thresholding}
In the above example, we assumed that the derivatives for the training data are known or can be computed with sufficient accuracy. If the derivatives, however, are noisy or inaccurate, the resulting matrix representations of the operators often become nonsparse and additional techniques such as denoising, total-variation regularization, or iterative hard thresholding might be required to eliminate spurious nonzero entries, see also~\cite{BrPrKu16} and references therein. In order to model the presence of noise, we replace $ b(x_l) $ by $ b(x_l) + \eta $, where $ \eta $ is sampled from a Gaussian distribution with standard deviation $ \varsigma $. By adding the iterative hard thresholding procedure proposed in~\cite{BrPrKu16} to gEDMD, which step by step removes entries larger than a given threshold $ \delta $ and then recomputes the coefficients, we can eliminate unwanted entries. The results, however, depend strongly on the chosen threshold as shown in Figure~\ref{fig:thresholding}. The smaller the signal-to-noise ratio, the larger the threshold needs to be to eliminate spurious nonzero entries, but a too large threshold will also eliminate the actual coefficients. The error here is defined to be the average difference between the true and the estimated coefficients after 10 iterations of the hard thresholding algorithm.

\begin{figure}
    \centering
    \includegraphics[width=0.45\textwidth]{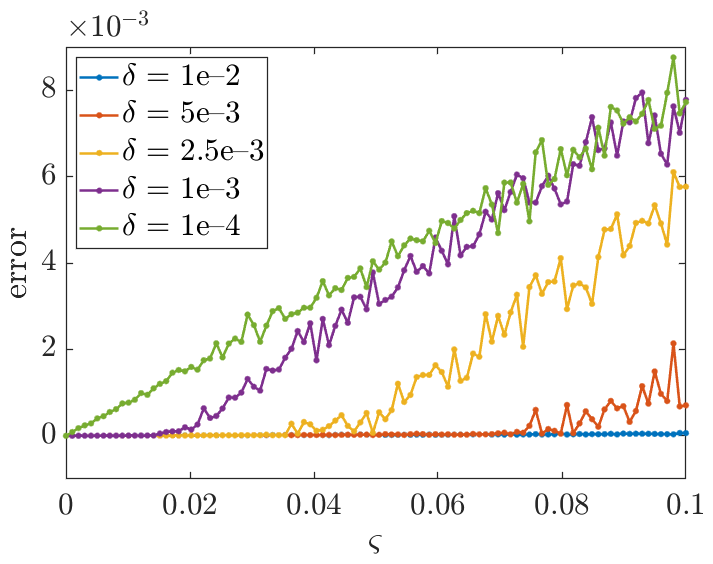}
    \caption{Recovery error as a function of the standard deviation $ \varsigma $ for different thresholds~$ \delta $. If no thresholding is used, the results coincide with the $ \delta = $ 1e--4 case. The results show that for inaccurate estimates of the derivatives additional techniques are required to obtain suitable representations of the system. Provided that the cut-off value is chosen judiciously, the hard thresholding approach enables us to recover the correct dynamics even in the presence of noise.}
    \label{fig:thresholding}
\end{figure}

\end{remark}

\subsubsection{Conservation laws}
\label{ssec:conservationdeterministic}

A function $E \colon \mathbb{R}^d \rightarrow \mathbb{R}$ is said to be a \emph{conserved quantity} if it remains constant for all~$t$ and all initial values, i.e., $\frac{\mathrm{d}}{\mathrm{d}t}E = \nabla E \cdot b = 0$, which immediately implies that $ E $ is an eigenfunction of the Koopman generator corresponding to the eigenvalue $ \lambda = 0 $; such invariants have already been considered in Koopman's original paper \cite{Ko31}. Similarly, eigenfunctions of the Perron--Frobenius generator associated with $ \lambda = 0 $ represent invariant densities. Conservation laws play an important role in physics and engineering, but are in principle hard to discover. The relationship between conservation laws and Koopman eigenfunctions has recently been exploited in \cite{Kaiser17, Kaiser18}, where conserved quantities are learned from data. In the same way, we can apply gEDMD to find non-trivial eigenfunctions corresponding to $ \lambda = 0 $.

\subsection{Non-deterministic dynamical systems}
\label{sec:gEDMD_stochastic}

Let us now extend these results to stochastic differential equations of the form \eqref{eq:SDE}. Given a set of training data $ \{\ts x_l \ts\}_{l=1}^m $ as above, we assume that $ \{\ts b(x_l) \ts\}_{l=1}^m $ and $ \{\ts \sigma(x_l) \ts\}_{l=1}^m $ are known or can be estimated.

\subsubsection{Generator approximation}

Let
\begin{equation} \label{eq:definition_dpsi_k}
    \mathrm{d}\psi_k(x) = (\mathcal{L} \psi_k)(x) = \sum_{i=1}^d b_i(x) \ts \pd{\psi_k}{x_i}(x) + \frac{1}{2} \sum_{i=1}^d \sum_{j=1}^d a_{ij}(x) \ts \pd{^2 \psi_k}{x_i \ts \partial x_j}(x)
\end{equation}
and
\begin{equation*}
    \mathrm{d}\Psi_X =
    \begin{bmatrix}
        \mathrm{d}\psi_1(x_1) & \dots  & \mathrm{d}\psi_1(x_m) \\
        \vdots                & \ddots & \vdots                \\
        \mathrm{d}\psi_k(x_1) & \dots  & \mathrm{d}\psi_k(x_m)
    \end{bmatrix}.
\end{equation*}
That is, in addition to the first derivatives of the basis functions, we now also need the second derivatives, which can again be precomputed analytically. Solving the resulting minimization problem, this leads to the least-squares approximation
\begin{equation*}
    M = \mathrm{d}\Psi_X \Psi_X^+ = \big(\mathrm{d}\Psi_X \Psi_X^\top\big) \big(\Psi_X \Psi_X^\top\big)^+ = \widehat{A} \ts \widehat{G}^+,
\end{equation*}
with
\begin{equation*}
    \widehat{A} = \frac{1}{m} \sum_{l=1}^m \mathrm{d}\psi(x_l) \ts \psi(x_l)^\top
    \quad \text{and} \quad
    \widehat{G} = \frac{1}{m} \sum_{l=1}^m \psi(x_l) \ts \psi(x_l)^\top.
\end{equation*}
As above, we obtain $ M = \widehat{L}^\top = \widehat{A} \ts \widehat{G}^+ $ as an empirical estimate of the generator and $ M^* = (\widehat{L}^*)^\top = \widehat{A}\ts^\top \widehat{G}^{+} $ as an estimate of the adjoint operator.

\begin{proposition}
\label{prop:convergence_gedmd}
In the infinite data limit, gEDMD converges to the Galerkin projection of the generator onto the space spanned by the basis functions $ \{\ts \psi_i \ts\}_{i=1}^n $.
\end{proposition}

\begin{proof}
The proof is equivalent to the counterpart for standard EDMD, see \cite{WKR15, KKS16}. Letting $ m $ go to infinity, we obtain
\begin{alignat*}{4}
    \widehat{A}_{ij} &= \frac{1}{m} \sum_{l=1}^m \mathrm{d}\psi_i(x_l) \ts \psi_j(x_l)
                &&\underset{\scriptscriptstyle m \rightarrow \infty}{\longrightarrow} \int (\mathcal{L} \psi_i)(x) \ts \psi_j(x) \ts \mathrm{d}\mu(x)
                &&= \innerprod{\mathcal{L} \psi_i}{\psi_j}_\mu &&= A_{ij}, \\
    \widehat{G}_{ij} &= \frac{1}{m} \sum_{l=1}^m \psi_i(x_l) \ts \psi_j(x_l)
                &&\underset{\scriptscriptstyle m \rightarrow \infty}{\longrightarrow} \int \psi_i(x) \ts \psi_j(x) \ts \mathrm{d}\mu(x) &&
                = \innerprod{\psi_i}{\psi_j}_\mu  &&= G_{ij},
\end{alignat*}
where $ x_l \sim \mu $. That is, the matrices $ \widehat{A} $ and $ \widehat{G} $ are empirical estimates of the matrices $ A $ and $ G $, respectively.
\end{proof}

\begin{remark}
If the drift and diffusion coefficients of the stochastic differential equation~\eqref{eq:SDE} are not known, they can be approximated via finite differences. In fact, by the Kramers--Moyal formulae,
\begin{alignat*}{2}
    b(x) &= \lim_{t \to 0} b^t(x) &&:= \lim_{t \to 0} \mathbb{E}\left[ \frac{1}{t} (X_t - x) \relmiddle| X_0 = x \right], \\
    a(x) &= \lim_{t \to 0} a^t(x) &&:= \lim_{t \to 0} \mathbb{E}\left[ \frac{1}{t} (X_t - x)(X_t - x)^\top \relmiddle| X_0 = x \right].
\end{alignat*}
These expressions can be evaluated pointwise by spawning multiple short trajectories from each data point $x_l$, and then estimating the expectations above via Monte Carlo. Alternatively, if a single ergodic simulation at time step $t$ is available, we can also replace the definition of $\mathrm{d}\psi_k$ in \eqref{eq:definition_dpsi_k} by
\begin{equation*}
\mathrm{d}\psi_k(x_l) = \frac{1}{t} (x_{l+1} - x_{l}) \cdot \nabla \psi_k(x_l) + \frac{1}{2\ts t}\left[(x_{l+1} - x_l)(x_{l+1} - x_l)^\top\right] : \nabla^2 \psi_k(x_l).
\end{equation*}
It was shown in \cite{BNC18} that in the infinite data limit
\begin{equation*}
\lim_{m \to \infty} \widehat{A}_{ij} = \innerprod{b^t \cdot \nabla \psi_i + \frac{1}{2} a^t : \nabla^2 \psi_i}{\psi_j}_\mu.
\end{equation*}
In this case, gEDMD converges to a Galerkin approximation of the differential operator with drift and diffusion coefficients $b^t$ and $a^t$.
\end{remark}

\begin{remark}
\label{rem:gedmd_reversible}
If the stochastic dynamics \eqref{eq:SDE} are reversible with respect to the measure~$\mu$, we only require first-order derivatives of the basis. In this case, the Galerkin matrix $A$ in~\eqref{eq:Galerkin_matrices} can be expressed as
\begin{equation*}
    A_{ij} = \innerprod{\mathcal{L}\psi_i}{\psi_j}_\mu = -\frac{1}{2}\int \nabla \psi_i \ts \sigma \ts \sigma^\top \nabla \psi_j^\top \, \mathrm{d}\mu,
\end{equation*}
where the drift coefficient enters only implicitly via the invariant measure $\mu$, see \cite{Zhang:2016aa}. Using the gradient matrix $\nabla \Psi \in \mathbb{R}^{n \times d}$, where each row corresponds to the gradient of a basis function, the empirical estimator $\widehat{A}$ for $A$ is then defined as follows:
\begin{align*}
\widehat{A} &= -\frac{1}{2m} \sum_{l=1}^m \mathrm{d}\psi(x_l) \ts \mathrm{d}\psi(x_l)^\top,
\end{align*}
with $ \mathrm{d}\psi(x_l) = \nabla \Psi(x_l) \ts \sigma(x_l) $.
\end{remark}

\begin{example}
Let us first compute eigenfunctions of the generator. We assume that $ \{\ts b(x_l) \ts\}_{l=1}^m $ and $ \{\ts \sigma(x_l) \ts\}_{l=1}^m $ are known and not estimated from data.
\begin{enumerate}[wide, itemindent=\parindent, itemsep=0ex, topsep=0.5ex]
\item We consider again the Ornstein--Uhlenbeck process defined in Example~\ref{ex:Ornstein Uhlenbeck}. For the numerical experiments, we set $ \alpha = 1 $ and $ \beta = 4 $ and select a basis comprising monomials of order up to and including ten. Using only $ 100 $ uniformly generated test points in $ \mathbb{X} = [-2, 2] $, we obtain the Koopman eigenfunctions shown in Figure~\ref{fig:OU_gEDMD}(a), which are virtually indistinguishable from the analytical solution. Standard EDMD would typically need more test points for such an accurate approximation of the dominant eigenfunctions, see~\cite{KNKWKSN18} for details.\!\footnote{Note that although the definition of the Ornstein--Uhlenbeck process is slightly different in \cite{KNKWKSN18}, the systems are in fact identical.} The results for the Perron--Frobenius generator using monomials are not as good, see Figure~\ref{fig:OU_gEDMD}(b). Replacing monomials by a basis containing Gaussian functions the results improve considerably as shown in Figure~\ref{fig:OU_gEDMD}(c). This illustrates that it is crucial to select suitable basis functions, which are, however, generally not known in advance. The sparsity patterns of the generator approximation using EDMD and gEDMD are compared in Figure~\ref{fig:OU_gEDMD}(d--f), showing that EDMD leads to less sparse matrices with additional spurious nonzero entries. \exampleSymbol

\begin{figure}
    \centering
    \begin{minipage}{0.32\textwidth}
        \centering
        \subfiguretitle{(a)}
        \includegraphics[width=\textwidth]{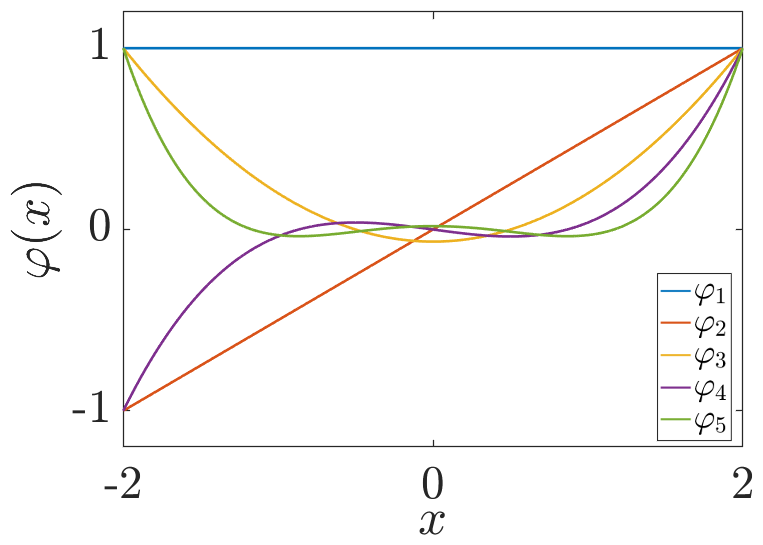}
    \end{minipage}
    \begin{minipage}{0.32\textwidth}
        \centering
        \subfiguretitle{(b)}
        \includegraphics[width=\textwidth]{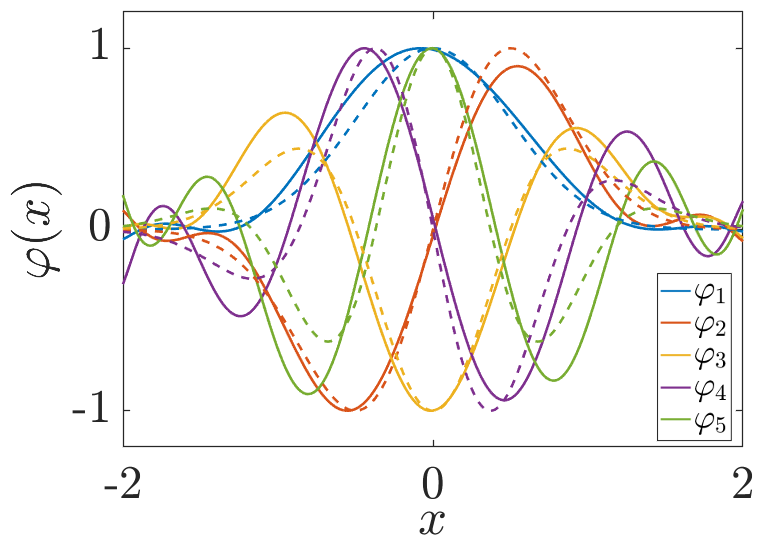}
    \end{minipage}
    \begin{minipage}{0.32\textwidth}
        \centering
        \subfiguretitle{(c)}
        \includegraphics[width=\textwidth]{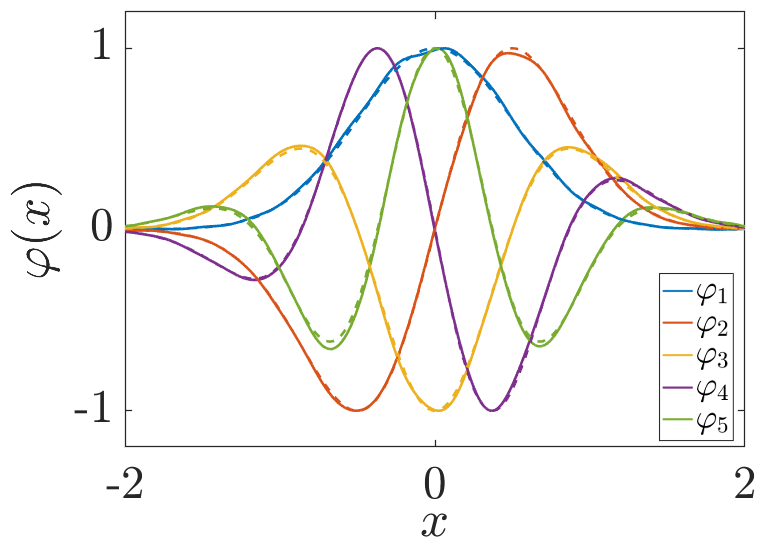}
    \end{minipage}
    \\[1ex]
    \begin{minipage}{0.32\textwidth}
        \centering
        \subfiguretitle{(d)}
        \includegraphics[width=0.95\textwidth]{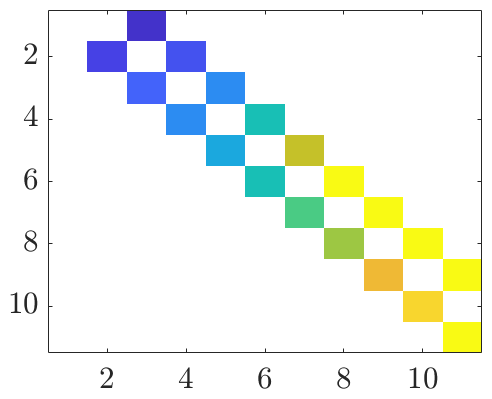}
    \end{minipage}
    \begin{minipage}{0.32\textwidth}
        \centering
        \subfiguretitle{(e)}
        \includegraphics[width=0.95\textwidth]{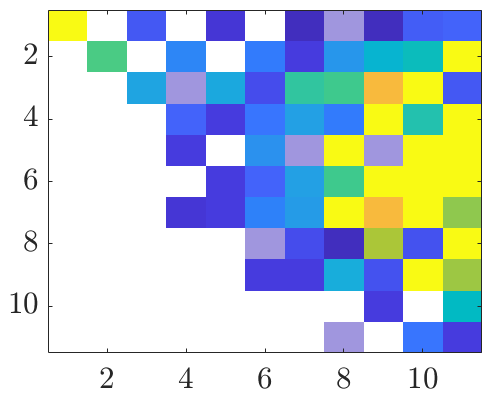}
    \end{minipage}
    \begin{minipage}{0.32\textwidth}
        \centering
        \subfiguretitle{(f)}
        \includegraphics[width=0.95\textwidth]{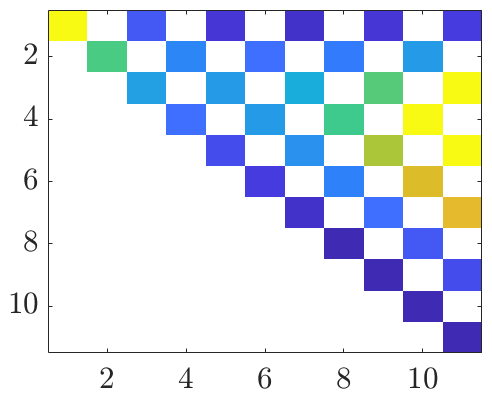}
    \end{minipage}
    \caption{Eigenfunctions of (a) the Koopman generator and (b) the Perron--Frobenius generator associated with the Ornstein--Uhlenbeck process computed using gEDMD with monomials of order up to ten. The dashed lines represent the analytically computed eigenfunctions. (c) Eigenfunctions of the Perron--Frobenius generator, where the basis now comprises 30 Gaussian functions. (d) Sparsity pattern of $ \widehat{L} $ computed with gEDMD, (e) sparsity pattern of $ \widehat{K}_\tau $ computed with EDMD, and (f) sparsity pattern of $ \exp(\tau \widehat{L}) $, where $ \tau $ is the lag time used for EDMD.}
    \label{fig:OU_gEDMD}
\end{figure}

\item We construct a more complicated example by defining $ V(x) = (x_1^2 - 1)^2 + x_2^2 $, which represents the renowned double-well potential, but then, instead of using isotropic noise, add a state-dependent diffusion term to obtain a stochastic differential equation of the form~\eqref{eq:SDE}, with
\begin{equation*}
    b(x) = -\nabla V(x) =
    \begin{bmatrix}
         4 \ts x_1 - 4 \ts x_1^3 \\
        -2 \ts x_2
    \end{bmatrix}
    \quad \text{and} \quad
    \sigma(x) =
    \begin{bmatrix}
        0.7 & x_1 \\
        0 & 0.5
    \end{bmatrix}.
\end{equation*}
The system exhibits metastable behavior, where the rare transitions are the jumps between the two wells. The potential and the two dominant eigenfunctions of the Perron--Frobenius generator computed with the aid of gEDMD are shown in Figure~\ref{fig:Double well}. Here, we generated $ 30\ts000 $ test points in $ \mathbb{X} = [-2, 2] \times [-1, 1] $ and selected a basis comprising $ 300 $ radial basis functions (whose centers are the midpoints of a regular box discretization) with bandwidth $ \sigma = 0.2 $. \exampleSymbol

\begin{figure}
    \centering
    \begin{minipage}{0.32\textwidth}
        \centering
        \subfiguretitle{(a)}
        \includegraphics[width=\textwidth]{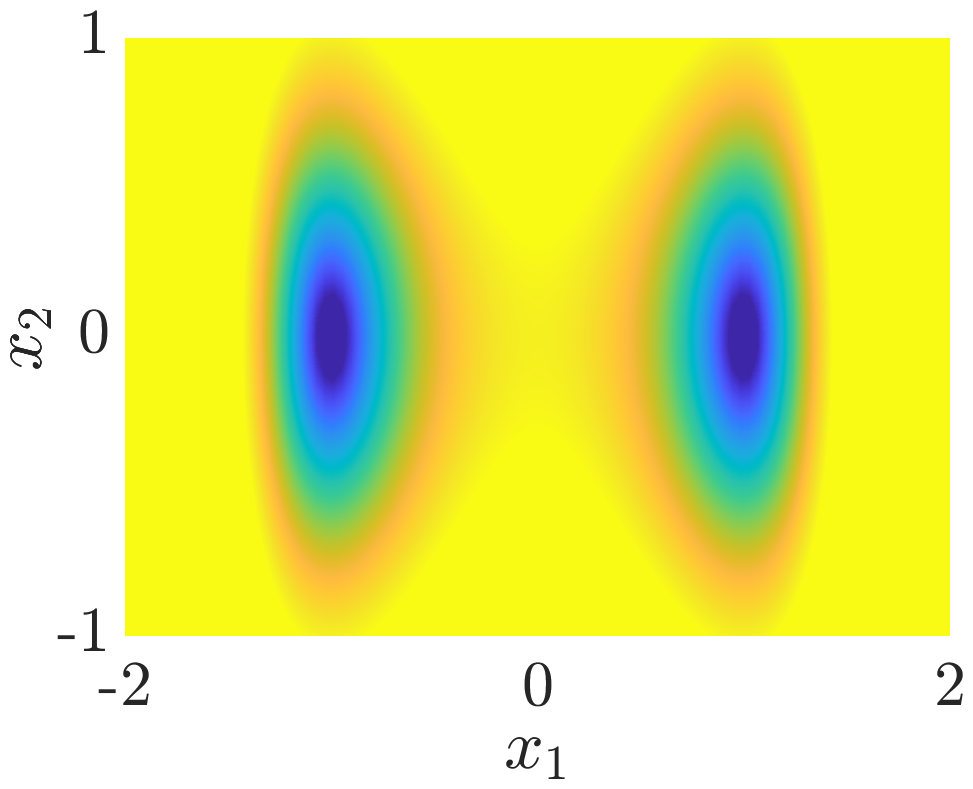}
    \end{minipage}
    \begin{minipage}{0.32\textwidth}
        \centering
        \subfiguretitle{(b)}
        \includegraphics[width=\textwidth]{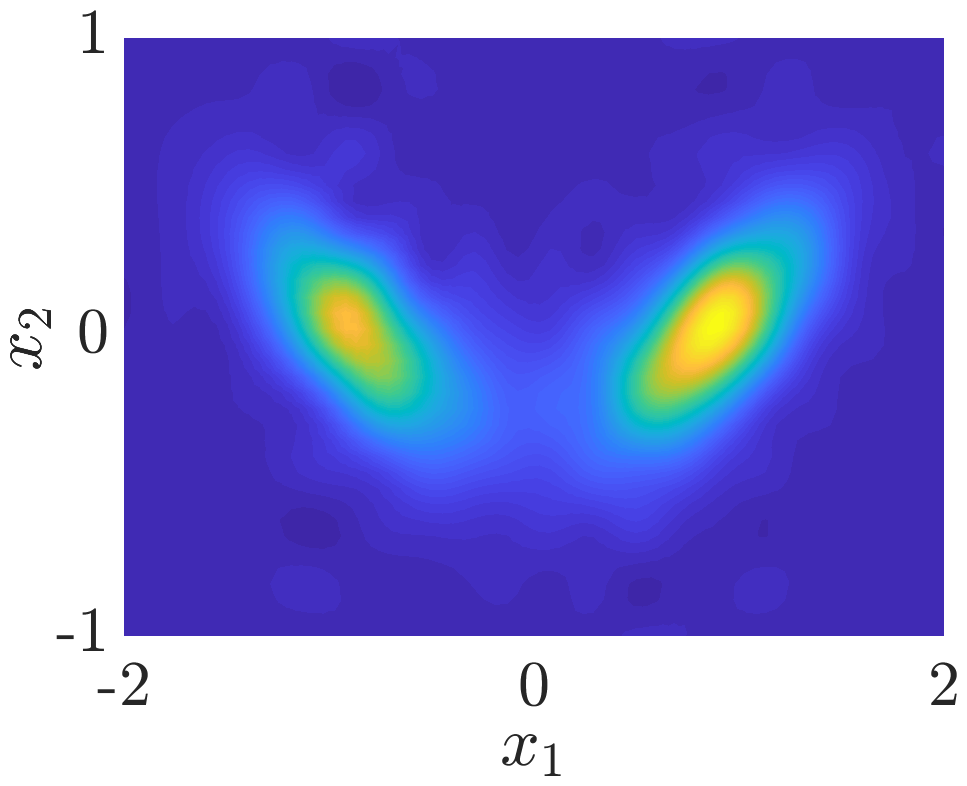}
    \end{minipage}
    \begin{minipage}{0.32\textwidth}
        \centering
        \subfiguretitle{(c)}
        \includegraphics[width=\textwidth]{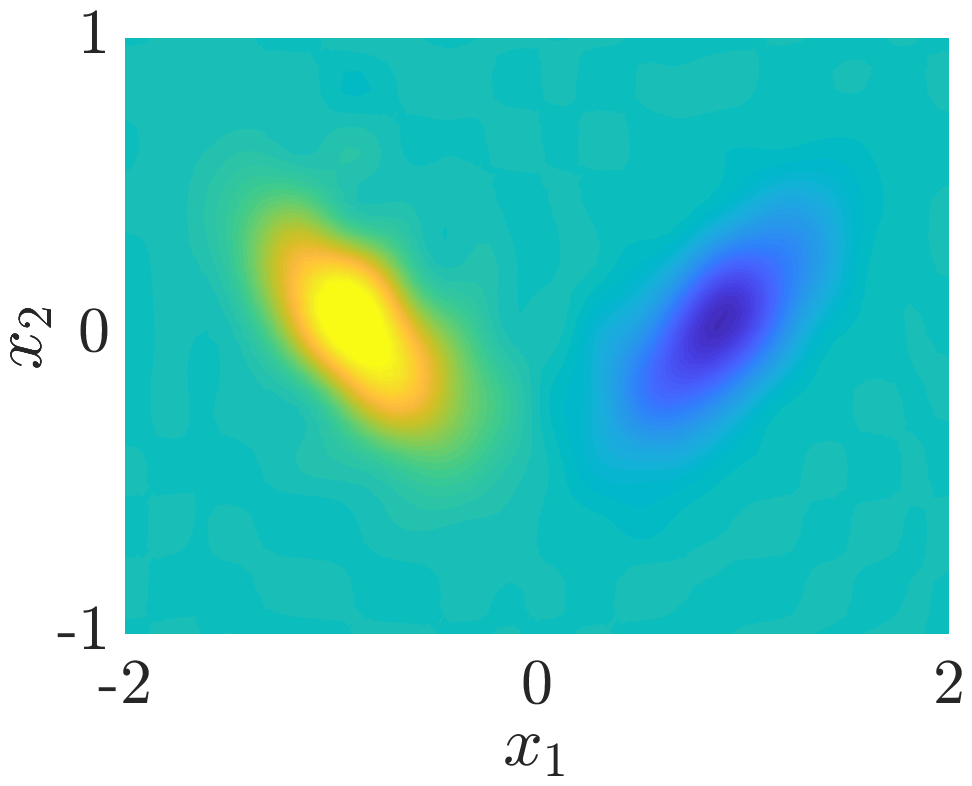}
    \end{minipage}
    \caption{(a)~Double-well potential. (b)~First and (c)~second eigenfunction of the Perron--Frobenius generator. Due to the non-isotropic noise the wells are tilted. The second eigenfunction clearly separates the two wells. In all plots, blue corresponds to small and yellow to large values.}
    \label{fig:Double well}
\end{figure}

\end{enumerate}

\end{example}

\subsubsection{System identification}
\label{sec:non_deterministic_sys_id}

As for deterministic systems, we can utilize the generator approximation also for system identification. In order to determine $ b $, we simply plug in the full-state observable $ g $ again. In addition to the drift term, we need to identify the diffusion term. This can be accomplished as follows: Note that for $ \psi_k(x) = x_i \ts x_j $, it holds that
\begin{equation*}
    (\mathcal{L} \psi_k)(x) = b_i(x) \ts x_j + b_j(x) \ts x_i + a_{ij}(x).
\end{equation*}
Since we already obtained a representation of $ b $ in the previous step, we can subtract the first two terms to obtain $ a_{ij} $. Here, we have to assume that both $ b_i $ and $ b_j $ can be written in terms of the basis functions and that, furthermore, also the functions multiplied by $ x_j $ or $ x_i $, respectively, are contained in the space spanned by $ \{\ts \psi_i \ts\}_{i=1}^n $. For instance, if $ b $ contains monomials of degree $ p $, then the dictionary must also contain monomials of degree $ p+1 $. For other types of basis functions, we have to make sure that the aforementioned requirement is satisfied as well.

\begin{example}
Let us illustrate the recovery of $ b $ and $ a $ from the generator representation.
\begin{enumerate}[wide, itemindent=\parindent, itemsep=0ex, topsep=0.5ex]
\item For the Ornstein--Uhlenbeck process, we immediately obtain $ b(x) = (\mathcal{L} \psi_2)(x) = -\alpha \ts x $ and $ a(x) = (\mathcal{L} \psi_3)(x) - 2 \ts b(x) \ts x = 2 \ts \beta^{-1} $, see the matrix representation of the generator in Example~\ref{ex:OrnUhl1}, which implies $ \sigma(x) = \sqrt{2 \ts \beta^{-1}} $. Thus, the system is identified correctly.

\item For the double-well problem, we generate $ 8000 $ random points in $ \mathbb{X} = [-2, 2] \times [-1, 1] $ and use the exact values for $ b(x) $ and $ \sigma(x) $. We then obtain an approximation of the generator whose first six columns for a dictionary comprising monomials up to order four are given by
\begin{equation*}
    \kbordermatrix{
          & 1 & x_1 & x_2 & x_1^2 & x_1 \ts x_2 & x_2^2 \\
        1               & 0 &  0 &  0 &  0.49 &  0   &  0.25 \\
        x_1             & 0 &  4 &  0 &  0    &  0.5 &  0    \\
        x_2             & 0 &  0 & -2 &  0    &  0   &  0    \\
        x_1^2           & 0 &  0 &  0 &  9    &  0   &  0    \\
        x_1 \ts x_2     & 0 &  0 &  0 &  0    &  2   &  0    \\
        x_2^2           & 0 &  0 &  0 &  0    &  0   & -4    \\
        x_1^3           & 0 & -4 &  0 &  0    &  0   &  0    \\
        x_1^2 \ts x_2   & 0 &  0 &  0 &  0    &  0   &  0    \\
        x_1 \ts x_2^2   & 0 &  0 &  0 &  0    &  0   &  0    \\
        x_2^3           & 0 &  0 &  0 &  0    &  0   &  0    \\
        x_1^4           & 0 &  0 &  0 & -8    &  0   &  0    \\
        x_1^3 \ts x_2   & 0 &  0 &  0 &  0    & -4   &  0    \\
        x_1^2 \ts x_2^2 & 0 &  0 &  0 &  0    &  0   &  0    \\
        x_1 \ts x_2^3   & 0 &  0 &  0 &  0    &  0   &  0    \\
        x_2^4           & 0 &  0 &  0 &  0    &  0   &  0
    }.
\end{equation*}
We can see that $ b $ is recovered correctly by the columns two and three. Furthermore, for the entries of the matrix $ a $, we obtain
\begin{alignat*}{2}
    a_{11}(x) &= (\mathcal{L}\psi_4)(x) - 2 \ts b_1(x) \ts x_1 &&= 0.49 + x_1^2, \\
    a_{12}(x) &= (\mathcal{L}\psi_5)(x) - b_1(x) \ts x_2 - b_2(x) \ts x_1 &&= 0.5 \ts x_1, \\
    a_{22}(x) &= (\mathcal{L}\psi_6)(x) - 2 \ts b_2(x) \ts x_2 &&= 0.25,
\end{alignat*}
which is indeed $ \sigma \sigma^\top $. Note that using only monomials of order up to three would allow us to recover $ b $ but not $ a $. \exampleSymbol
\end{enumerate}
\end{example}

\begin{remark}
It is worth noting that:
\begin{enumerate}[wide, itemindent=\parindent, itemsep=0ex, topsep=0.5ex]
\item Although we presented only systems composed of monomials (mainly for the sake of illustration), the proposed method allows for arbitrary dictionaries containing twice continuously differentiable functions.
\item We identify $ a = \sigma \ts \sigma^\top $ and not $ \sigma $ itself. If it is necessary to evaluate $ \sigma $, e.g., when using the identified system to generate new trajectories, we can obtain it, for instance, by a Cholesky decomposition of $ a $, see also \cite{Zhang:2016aa}. Note, however, that $ \sigma $ is not uniquely defined.
\item The method relies on accurate estimates of the drift and diffusion terms. Noisy data will lead to nonsparse solutions, which can then be improved by applying iterative hard thresholding again, see Remark~\ref{rem:thresholding}. We now add noise with variance $ \varsigma = 0.1 $ to the drift and diffusion terms. After sparsifying the estimated matrix approximation of the Koopman generator with a threshold $ \delta = 0.1 $, we obtain
\begin{equation*}
    b(x) = \begin{bmatrix}
       4.00057 \ts x_1 - 4.00012 \ts x_1^3 \\
       -1.99998 \ts x_2 \\
    \end{bmatrix}
\end{equation*}
and
\begin{align*}
    a_{11}(x) &= 0.50035 + 0.99901 \ts x_1^2 - 0.00016 \ts x_1^4, \\
    a_{12}(x) &= 0.49729 \ts x_1 - 0.00250 \ts x_1 \ts x_2 + 0.00097 \ts x_1^3 \ts x_2, \\
    a_{22}(x) &= 0.25648 + 0.00720 \ts x_2^2.
\end{align*}
Note that the noise is picked up by the diffusion term which might thus be overestimated. Nevertheless, the coefficients are still close to the exact solution. Instead of eliminating small coefficients of $ (\mathcal{L}\psi_k)(x) $, we could apply iterative hard thresholding to the coefficients of $ a_{ij}(x) $ to find a parsimonious representation of $ a(x) $.
\end{enumerate}
\end{remark}

This method to discover the drift and diffusion terms of stochastic differential equations suffers from the same shortcomings as SINDy: The validity of the learned model depends crucially on whether or not both $ b $ and $ a $ can be expressed in terms of the basis functions and also on the availability of accurate estimates of the derivatives. Ideally, the resulting model is parsimonious, minimizing model complexity while simultaneously enabling accurate predictions without overfitting. Nonsparse solutions typically indicate that the expressivity of the dictionary is not sufficient or that the data is too noisy. Adding more basis functions or increasing the size of the data set might alleviate such problems. However, positing that the model comprises only a few simple terms, the method presented here allows for the identification of the governing equations of stochastic dynamical systems. Additionally, the approximation of the generator is an important problem in itself. The eigenvalues and eigenfunctions contain information about time scales and metastable sets and can be used for model reduction and control. This will be described in more detail in Section~\ref{sec:Further applications}.

\subsubsection{Conservation laws}

If $E$ is a conserved quantity of a non-deterministic system, then the definition of the Koopman operator \eqref{eq:stochastic_koopman_op} and the partial differential equation $ \pd{u}{t} = \mathcal{L} u $ imply that $\mathcal{L} E = 0$, just as in the deterministic case. Hence, conserved quantities can also be approximated by extracting non-trivial eigenfunctions associated with $\lambda = 0$ using gEDMD. The same precautions as discussed in Section \ref{ssec:conservationdeterministic} apply.

\begin{remark}
For a stochastic dynamical system in the sense of Stratonovich, i.e.,
\begin{equation*}
    \mathrm{d}X_t = b(X_t) \ts \mathrm{d}t + \sigma(X_t) \circ \mathrm{d}W_t,
\end{equation*}
a sufficient condition for $E$ to be conserved is
\begin{equation*}
     \nabla E^\top \left [ b  + \sum_{i=1}^s \sigma_i \right] = 0,
\end{equation*}
which is similar to the deterministic case. Here, $\sigma_i$ denotes the $i$th column of $\sigma$. This result follows directly from the chain rule of Stratonovich calculus, see \cite{FaLe09, ZhZhHoSo16}.
\end{remark}

\begin{example}
Consider the noisy Duffing oscillator, i.e., for $\alpha, \beta, \varepsilon \in \mathbb{R}$ we have a Stratonovich stochastic differential equation with
\begin{equation*}
    b(x) =
    \begin{bmatrix}
        x_2 \\
        -\alpha x_1 -\beta x_1^3
    \end{bmatrix}
    \quad \text{and} \quad
    \sigma(x) = \varepsilon \ts b(x).
\end{equation*}
To apply gEDMD, we convert it to an It\^o stochastic differential equation using the drift correction formula to correct the noise-induced drift, which is defined componentwise as
\begin{equation*}
    c_i(x) = \sum_{j=1}^d \sum_{k = 1}^s  \frac{\partial \sigma_{ik}}{\partial x_j}(x) \ts \sigma_{jk}(x), \quad i = 1,\dotsc, d,
\end{equation*}
see \cite{St66}.
We obtain the It\^o stochastic differential equation
\begin{equation*}
        \mathrm{d}X_t = \big(b(X_t) + \tfrac{1}{2} \ts c(X_t)\big) \ts \mathrm{d}t + \sigma(X_t) \ts \mathrm{d}W_t \quad\text{with}\quad c(x) = \varepsilon^2
    \begin{bmatrix}
    b_2(x) \\
    \left(-\alpha  -3\ts\beta x_1^2\right) b_1(x)
    \end{bmatrix}.
\end{equation*}
Setting $\alpha = -1.1$, $\beta = 1.1$, $\varepsilon = 0.05$, choosing a dictionary that contains monomials, and applying gEDMD, the multiplicity of the eigenvalue $\lambda = 0$ is two and we obtain a conserved quantity of the form
\begin{equation*}
    E(x) \approx \tfrac{\alpha}{2} x_1^2 + \tfrac{\beta}{4} x_1^4 + \tfrac{1}{2}x_2^2 + c,
\end{equation*}
where $c \in \mathbb{R}$ is an arbitrary constant. \exampleSymbol
\end{example}

\subsection{Relationships with other methods}
\label{ssec:Relationships with other methods}

We will now point out similarities and differences between the methods presented above and other well-known approaches for systems identification and generator approximation.

\subsubsection{SINDy}

SINDy \cite{BrPrKu16} was designed to learn ordinary differential equations from simulation or measurement data. Just like gEDMD, it requires a set of states and the corresponding time derivatives. Defining $ \dot{X} = [\dot{x}_1, \dot{x}_2, \dots, \dot{x}_m] $, SINDy minimizes the cost function $ \norm{\smash{\dot{X} - M_{\scriptscriptstyle S} \Psi_X}}_F$, i.e., $ M_{\scriptscriptstyle S} = \dot{X} \Psi_X^+ $. Here, we omit the sparsification constraints, which can be added in the same way to gEDMD as described above. Recall that we assume that the full-state observable is given by $ g(x) = B^\top \psi(x) $. SINDy can thus be seen as a special case of gEDMD since
\begin{equation*}
    \dot{x} = B^\top \dot{\psi}(x)
            \approx B^\top M \psi(x)
            = \underbrace{B^\top \dot{\Psi}_X}_{\dot{X}} \Psi_X^+ \psi(x)
            = \underbrace{\dot{X} \Psi_X^+}_{M_{\scriptscriptstyle S}} \psi(x)
            = M_{\scriptscriptstyle S} \ts \psi(x).
\end{equation*}

\subsubsection{Koopman lifting technique}

The Koopman lifting technique \cite{MauGon16, MauGon17} uses the infinitesimal generator $ \mathcal{L} $ for system identification. While tailored mainly to ordinary differential equations, extensions to stochastic differential equations with isotropic noise are also considered. First, the Koopman operator for a fixed lag time $ \tau $ is estimated from trajectory data with the aid of standard EDMD. Then an approximation of the generator is obtained by taking the matrix logarithm, i.e.,
\begin{equation*}
    \widehat{L} = \tfrac{1}{\tau} \log \widehat{K}_\tau,
\end{equation*}
where $ \widehat{K}_\tau $ is the matrix representation of the Koopman operator with respect to the chosen basis $ \psi $ (and lag time $ \tau $). The last step is to estimate the governing equations in the same way as illustrated in Example~\ref{ex:systemidentification} for gEDMD. The Koopman lifting technique does not require the time-derivatives of the states or the partial derivatives of the basis functions, but only pairs of time-lagged data. However, the non-uniqueness of the matrix logarithm can cause problems and a sufficiently small sampling time $ \tau $ is needed to ensure that the (possibly complex) eigenvalues lie in the strip $ \left\{ z \in \mathbb{C} : \vert \Im(z) \vert < \pi \right\} $, where $ \Im $ denotes the imaginary part. Roughly speaking, only an infinite sampling rate allows us to capture the entire spectrum of frequencies \cite{MauGon17}. Our approach generalizes to arbitrary systems of the form \eqref{eq:SDE}, but the estimation of the diffusion term can be carried over to the lifting technique as well. This could be a valuable alternative, e.g., when only trajectory data is available. If the exact derivatives for the training data are known, then gEDMD is in general more accurate than the lifting approach. If, on the other hand, the derivatives for gEDMD have to be approximated from trajectory data, then the accuracy depends on the order of the finite-difference approximation and the step size, while the accuracy of the lifting approach depends on the lag time and the matrix logarithm implementation.

\subsubsection{KRONIC}

KRONIC \cite{Kaiser17, Kaiser18}, which stands for \emph{Koopman reduced order nonlinear identification and control}, is a data-driven method for discovering Koopman eigenfunctions, which are then used for control and the detection of conservation laws. The approach is based on SINDy and assumes that an eigenvalue is known a priori (or simultaneously learns the eigenvalue and corresponding eigenfunction). In our notation, the resulting problem can be written as
\begin{equation*}
    \left( \lambda_\ell \Psi_X^\top - \dot{\Psi}_X^\top \right) \xi_\ell = 0,
\end{equation*}
which, multiplying from the left by $ \Psi_X $ and assuming that $ \Psi_X \Psi_X^\top $ is regular, becomes the gEDMD eigenvalue problem. This operator formulation is briefly mentioned in \cite{Kaiser17} as well. Thus, for deterministic systems, despite their different derivations, gEDMD and KRONIC are strongly related.

\section{Further applications}
\label{sec:Further applications}

In addition to identifying fast and slow modes, governing equations, or conservation laws, the Koopman generator has further applications that we will briefly demonstrate.

\subsection{Coarse-grained dynamics and gEDMD}

\subsubsection{Galerkin approximation}
\label{sec:cg_gedmd}

In what follows, we describe how models of the Koopman generator can be used to identify reduced order models of a (possibly high-dimensional) stochastic dynamical system. To get started, we recapitulate the model reduction formalism introduced by \cite{Legoll:2010aa,Zhang:2016aa}.
Assume the stochastic process given by \eqref{eq:SDE} possesses a unique invariant density $\mu$, and let $\xi \colon \R^d \to \R^p$ be a coarse-graining function which maps $\R^d$ to a lower-dimensional space $\R^p$. The coarse-graining map induces a reduced probability measure with density $\nu$ on $\R^p$. Consider the space $L^2_\nu$ of square-integrable functions of the reduced variables $z$. In fact, $L^2_\nu$ is an infinite-dimensional subspace of $L^2_\mu$, if each function $f \in L^2_\nu$ is identified with the function $f \circ \xi \in L^2_\mu$. 
Let $\mathcal{P}$ be the orthogonal projection onto $L^2_\nu$. Define a coarse-grained generator as
\begin{equation} \label{eq:projected_generator}
    \mathcal{L}^\xi = \mathcal{P}\mathcal{L}\mathcal{P}.
\end{equation}
Given suitable assumptions on the original process \eqref{eq:SDE}, $\mathcal{L}^\xi$ is again the infinitesimal generator of a stochastic dynamics on $\mathbb{R}^p$, with invariant density $\nu$ and effective drift and diffusion coefficient $b^\xi,\, a^\xi$ \cite{Legoll:2010aa,Zhang:2016aa}. 

First, we show that for a basis set of functions defined only on $\mathbb{R}^p$, gEDMD converges to a Galerkin approximation of the coarse-grained generator $\mathcal{L}^\xi$:
\begin{proposition}
Let $\mathbb{V} = \mathrm{span}\{\psi_k \}_{k=1}^n$ be a subspace of $L^2_\nu$. Then gEDMD applied to the functions $\widetilde{\psi}_k = \psi_k \circ \xi$ converges to the Galerkin projection of $\mathcal{L}^\xi$ onto $\mathbb{V}$. Here, \eqref{eq:definition_dpsi_k} needs to be updated by
\begin{align*}
\mathrm{d}\widetilde{\psi}_k(x) &= b(x) \ts \nabla_x \ts \xi(x) \ts \nabla_z \psi_k^\top(\xi(x)) + \frac{1}{2}(a(x) : H^\xi(x)) \ts \nabla_z \psi_k^\top(\xi(x)) \\
& ~~+ \frac{1}{2 }\nabla^2_z \psi_k(\xi(x)) : \left[ \nabla \xi(x)^\top a(x) \nabla \xi(x)\right],
\end{align*}
where $\nabla_x \ts \xi \in \mathbb{R}^{d \times p}$ is the Jacobian of $\xi$, and $H^\xi \in \mathbb{R}^{d \times d \times p}$ is the tensor of Hessian matrices for each component of $\xi$. The Frobenius inner product between $a$ and $H^\xi$ is applied to the first two dimensions of $H^\xi$.
\end{proposition}

\begin{proof}
The expression for $\mathrm{d}\widetilde{\psi}_k(x)$ follows from the chain rule. It was already shown in \cite{Zhang:2016aa} that
\begin{equation*} 
    \begin{split}
        \innerprod{\psi_i}{\psi_j}_\nu &= \innerprod{\psi_i\circ \xi}{\psi_j\circ \xi}_\mu, \\
        \innerprod{\smash{\mathcal{L}^\xi \psi_i}}{\psi_j}_\nu &= \innerprod{\mathcal{L}(\psi_i\circ \xi)}{\psi_j\circ \xi}_\mu.
    \end{split}
\end{equation*}
Thus, Proposition \ref{prop:convergence_gedmd} implies that
\begin{alignat*}{2}
\widehat{A}_{ij} &\rightarrow \innerprod{\mathcal{L}\widetilde{\psi}_i}{\widetilde{\psi}_j}_\mu &&= \innerprod{\mathcal{L}^\xi \psi_i}{\psi_j}_\nu, \\
\widehat{G}_{ij} &\rightarrow \innerprod{\widetilde{\psi}_i}{\widetilde{\psi}_j}_\mu &&= \innerprod{\psi_i}{\psi_j}_\nu. \tag*{\qedhere}
\end{alignat*}
\end{proof}

In summary, data of the original process, sampling the distribution $\mu$, can be used to learn a matrix representation of the coarse-grained generator \eqref{eq:projected_generator}. This matrix approximation can then be used to perform system identification, simulation, and control of the coarse-grained system the same way as described in Section \ref{sec:gEDMD_stochastic}.

\subsubsection{Separate Identification}
For a reversible stochastic differential equation~\eqref{eq:SDE}, we present an alternative approach to identify the parameters of the corresponding coarse-grained system. The method is related to spectral matching as introduced in \cite{NBC19}. The authors of~\cite{Zhang:2016aa} have shown that reversibility of the full process implies the dynamics generated by~\eqref{eq:projected_generator} are also reversible. Recalling that reversible dynamics are characterized by a scalar potential and the diffusion field, the basic idea is simply to estimate these two terms separately. The resulting framework, which we will call \textit{separate identification}, consists of four steps, which are only partially dependent on each other:

\paragraph{Force Matching}
The scalar potential $F^\xi$ of the coarse-grained dynamics can be estimated by an established technique called \textit{force matching} \cite{IZEKOV2005,NOID2008}. It is based on the fact that the gradient of $F^\xi$ solves the following minimization problem \cite{CLV2008}:
\begin{align}
\label{eq:force_matching_problem}
\nabla_z F^\xi &= \mathrm{argmin}_{g\in (L^2_\nu)^p} \int_{\mathbb{R}^d} \|
g(\xi(x)) - f^\xi_{lmf}(x)\|^2\,\mathrm{d}\mu(x), \\
\label{eq:definition_lmf}
f^\xi_{lmf} &= -\nabla_x F \cdot  G^\xi + \nabla_x \cdot G^\xi, \\
G^\xi &= \nabla_x \xi \left[(\nabla_x \xi)^T \nabla_x \xi\right]^{-1},
\end{align}
where the minimization is over all square-integrable vector fields $g$ of the reduced variables~$z$, and the divergence is applied separately to each column of $G^\xi$ in (\ref{eq:definition_lmf}). The vector field $f^\xi_{lmf}$ is called \textit{local mean force}, while $F$ is the scalar potential of the full process.

\paragraph{Application of gEDMD}
The second step consists of applying gEDMD to estimate a finite-dimensional model of the coarse-grained generator $\mathcal{L}^\xi$ as described in Section \ref{sec:cg_gedmd}, using a basis of functions $\{\ts \psi_i \ts\}_{i=1}^n$ defined on the reduced space $\R^p$. In particular, we obtain an estimate of the Galerkin matrix
\begin{align*}
    \widehat{A}_{ij} &= \innerprod{\mathcal{L}^\xi \psi_i}{\psi_j}_\nu.
\end{align*}

\paragraph{Learning the diffusion field}
As already discussed in Remark \ref{rem:gedmd_reversible}, matrix elements of the reduced generator are given by
\begin{equation}
\label{eq:generator_matrix_reversible_2}
\innerprod{\mathcal{L}^\xi \psi_i}{\psi_j}_\nu = -\frac{1}{2} \int \nabla_z \psi_i\, a^\xi \, \nabla_z \psi_j\,\mathrm{d}\nu.
\end{equation}
It follows that the effective diffusion can be learned by matching it to the generator matrix $\widehat{A}$ via \eqref{eq:generator_matrix_reversible_2}. Let $a^\xi(\theta)$ be a parametric model for the effective diffusion. Then the optimal set of parameters can be found by minimizing the Frobenius norm error
\begin{align} \label{eq:minimization_problem_diffusion}
E(\theta) &= \|\hat{A} - A(\theta) \|^2_F, \\
A(\theta)_{ij} &= -\frac{1}{2} \int \nabla_z \psi_i \, a^\xi(\theta) \, \nabla_z \psi_j\,\mathrm{d}\nu.
\end{align}

\paragraph{Determination of the drift}
Using the relationship between drift and diffusion of a reversible system, inserting estimates for $F^\xi$ and $a^\xi$ into \eqref{eq:rev_drift_diffusion_potential} below completes the definition of the reduced model
\begin{equation} \label{eq:rev_drift_diffusion_potential}
    b^\xi = -\frac{1}{2} a^\xi \ts \nabla F^\xi + \frac{1}{2} \nabla \cdot a^\xi.
\end{equation}
The above formulation seems advantageous compared to the direct system identification described in Section \ref{sec:non_deterministic_sys_id} for several reasons:

\begin{itemize}[wide, itemindent=\parindent, itemsep=0ex, topsep=0.5ex]
\item Separate basis sets can be used to calculate the Galerkin matrix, the potential, and the diffusion. Specifically, constraints on each of these (such as positive definiteness of the diffusion) can be incorporated into each basis individually. Moreover, learning of the potential and the diffusion can also be accomplished using nonlinear models.
\item The coordinate functions $z_i$ and $z_i \ts z_j$, as well as the products of the coordinate functions with the effective drift, are no longer required to be contained in the basis set.
\item Both force matching and \eqref{eq:minimization_problem_diffusion} are regression problems, allowing for the use of model validation techniques like cross-validation.
\item The dynamics obtained by combining the learned potential and diffusion via \eqref{eq:rev_drift_diffusion_potential} are automatically reversible.
\item By diagonalizing the generator matrix corresponding to $A(\theta)$ above, the spectrum of the learned dynamics can be calculated directly and compared to the spectrum of the generator matrix corresponding to $A$, providing a further means of model validation.
\end{itemize}

On the other hand, the direct system identification is more general since the reconstruction via the local mean force may fail to yield good approximations of the effective drift in cases where some parts of the dynamics orthogonal to the low-dimensional manifold defined by the reaction coordinate are slow. 

\subsubsection{Example 1: Lemon-slice potential}

We consider overdamped Langevin dynamics (see Remark \ref{rem:overdamped_langevin}) at inverse temperature $\beta = 1$ in the following two-dimensional potential $V$, expressed in polar coordinates:
\begin{equation*}
    V(r, \varphi) = \cos(k \ts \varphi) + \sec(0.5 \ts \varphi) + 10(r-1)^2 + \frac{1}{r}.
\end{equation*}
For $k = 4$, a contour of the potential is shown in Figure~\ref{fig:2d_lemon_slice}(a) below. Because of the two singular terms, the system's state space does not include the set $\{(x_1, x_2):\, x_1 \leq 0,\, x_2 = 0 \}$, enabling us to map the two-dimensional state space to polar coordinates unambiguously.

The polar angle $\varphi$ is a suitable reaction coordinate, so we choose $\xi(x_1, x_2) = \varphi(x_1, x_2)$. Due to the simplicity of the system, all relevant quantities can be calculated analytically. Using the full-state partition function $Z$ and two numerical constants $C_1, C_2$, see \cite{Nueske2019}, the invariant distribution, the effective drift and the effective diffusion along $\varphi$ are given by
\begin{align*}
\nu(\varphi) &= \frac{C_2}{Z} \exp\left(-\left[\cos(k \ts \varphi) + \sec(0.5 \ts \varphi)  \right]\right), \\
b^\varphi(\varphi) &= \frac{C_1}{C_2}\left[k\sin(k \ts \varphi) - 0.5 \tan(0.5 \ts \varphi)\sec(0.5 \ts \varphi)\right], \\
a^\varphi(\varphi) &= \frac{2 \ts C_1}{C_2}.
\end{align*}

We apply coarse-grained gEDMD with a basis set of Legendre polynomials up to degree~20, scaled to fit the domain $[-\pi,\pi]$. In this example and the next, we use exact expressions for the drift and diffusion coefficient, as the parameters of the full system are usually known in the context of model reduction. From the generator matrix, we obtain estimates of the effective drift and diffusion as described in Section \ref{sec:non_deterministic_sys_id}. Moreover, we also apply separate identification to learn the scalar potential and the diffusion. To this end, we use a basis set of periodic Gaussian functions centered at equidistant points between $\varphi = -2.8$ and $\varphi = 2.8$. The bandwidth of these Gaussians is determined by cross-validation, and is found to be $0.1$ for force matching and $2.0$ for the diffusion. We also enforce positivity of the diffusion by applying positivity constraints to the regression problem \eqref{eq:minimization_problem_diffusion}. We see in Figure~\ref{fig:2d_lemon_slice}(b) and \ref{fig:2d_lemon_slice}(c) that both methods provide accurate representations of the effective parameters. However, the diffusion estimated from \eqref{eq:minimization_problem_diffusion} is virtually indistinguishable from the analytical solution, while the representation obtained from the polynomial basis is more oscillatory.

We also verify that gEDMD correctly captures the slow dynamics in this example. We diagonalize the generator matrix obtained from the polynomial basis, and compute the first three implied time scales by taking reciprocals of the first three nontrivial eigenvalues (leaving out the zero eigenvalue). We compare these time scales to those extracted from a Markov state model (MSM) \cite{Prinz2011} inferred directly from the data. We find in Figure~\ref{fig:2d_lemon_slice}(d) that the time scales are in very good agreement. As described above, we also use the generator matrix corresponding to the optimal $ A(\theta) $ to estimate the first three implied time scales, and find them to match almost perfectly as well.

\begin{figure}
    \centering
    \begin{minipage}[t]{0.49\textwidth}
        \centering
        \subfiguretitle{(a)}
        \vspace*{0.6ex}
        \includegraphics[width=0.87\textwidth]{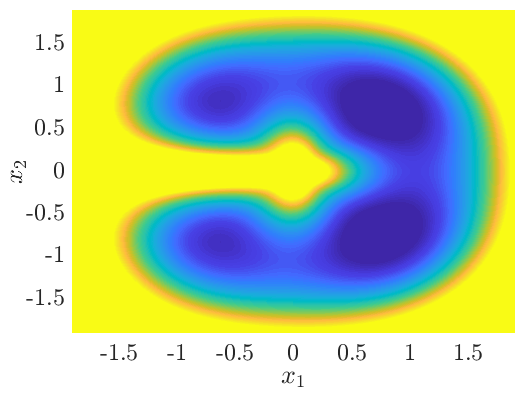}
    \end{minipage}
    \begin{minipage}[t]{0.49\textwidth}
        \centering
        \subfiguretitle{(b)}
        \includegraphics[width=0.9\textwidth]{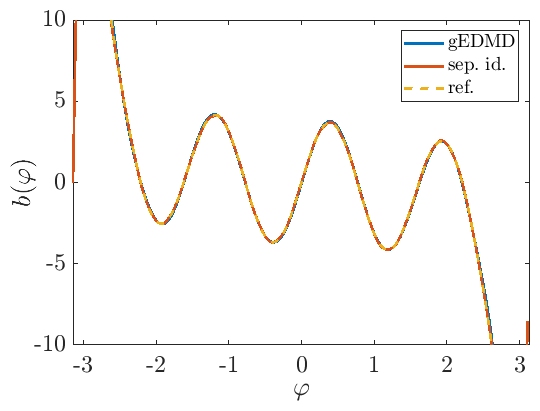}
    \end{minipage} \\[0.5ex]
    \begin{minipage}[t]{0.49\textwidth}
        \centering
        \subfiguretitle{(c)}
        \includegraphics[width=0.9\textwidth]{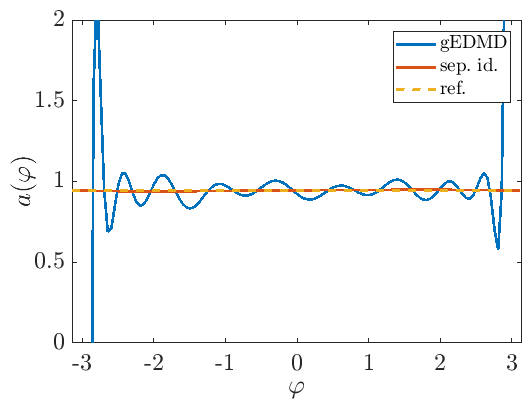}
    \end{minipage}
    \begin{minipage}[t]{0.49\textwidth}
        \centering
        \subfiguretitle{(d)}
        \includegraphics[width=0.9\textwidth]{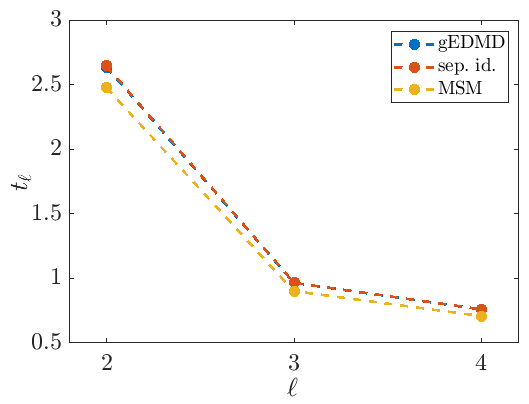}
    \end{minipage}
    \caption{Application of gEDMD with 21 Legendre polynomials to one-dimensional coarse-graining of the two-dimensional lemon-slice potential. (a) Visualization of the potential. (b)~Estimates of the effective drift along the polar angle $\varphi$ obtained directly from the generator matrix (blue), and by combining the solutions of \eqref{eq:minimization_problem_diffusion} and \eqref{eq:force_matching_problem} via \eqref{eq:rev_drift_diffusion_potential} (red). The analytical reference is shown in yellow. (c) Estimates of the effective diffusion along the polar angle $ \varphi $ obtained directly from the generator matrix (blue), and by learning the diffusion via \eqref{eq:minimization_problem_diffusion} using a Gaussian basis set (red), compared to the analytical reference in yellow. (d)~Estimates of the three slowest implied time scales using a Markov state model (yellow), diagonalization of the generator matrix (blue), and diagonalization of the generator matrix corresponding to the optimal diffusion (red).}
    \label{fig:2d_lemon_slice}
\end{figure}

\subsubsection{Example 2: Alanine dipeptide}

As a more complex example, we derive a coarse-grained model from molecular dynamics simulations of alanine dipeptide, which has been used as a test case in numerous previous studies. The data set is the same as in reference \cite{Wang2019} and comprises one million snapshots of Langevin dynamics saved every $1\,\mathrm{ps}$. As is well-known, the positional component of Langevin dynamics behaves approximately like an overdamped process, see Remark \ref{rem:overdamped_langevin}, up to a re-scaling of time. Hence, we apply gEDMD assuming the original process is overdamped, and we extract this effective unit of time by comparing the first two implied time scales obtained from gEDMD and from a reference Markov state model.

The slowest dynamics of alanine dipeptide are captured by a single internal molecular coordinate, called $\phi$-dihedral angle, which we choose to be the coarse-graining coordinate. Figure \ref{fig:ala2}(a) shows the empirical coarse-grained energy $F^\phi$, and an approximation obtained by applying force matching. The basis set for force matching consists of 57 periodic Gaussians of bandwidth 1.2, centered at equidistant points between -2.8 and 2.8. The slowest dynamical process corresponds to the transition across the highest barrier in this energy landscape.

We apply gEDMD with the first 26 Legendre polynomials scaled to fit the domain $[-2.7, 2.7]$. From the generator matrix, we extract the effective drift and diffusion, which are depicted by blue lines in Figures \ref{fig:ala2}(b) and \ref{fig:ala2}(c). As a comparison, we also compute an estimate of the effective diffusion by minimization of \eqref{eq:minimization_problem_diffusion}, including positivity constraints, with a set of 29 Gaussians of bandwidth 0.8, where the optimal bandwidth was determined by cross-validation. The resulting estimate of the diffusion is far less oscillatory than the direct estimate using the generator matrix, while the corresponding drift obtained from~\eqref{eq:rev_drift_diffusion_potential} is similar to the direct estimate.

Finally, we verify that gEDMD accurately reproduces the spectral properties of the original dynamics. As we can see in Figure \ref{fig:ala2}(d), after re-scaling the first two time scales provided by gEDMD by the effective time unit described above, they agree well with the results of an MSM analysis. The same is true for the time scales calculated based on the generator matrix corresponding to the optimal diffusion obtained by solving \eqref{eq:minimization_problem_diffusion}.

\begin{figure}[htb]
    \centering
    \begin{minipage}[t]{0.49\textwidth}
        \centering
        \subfiguretitle{(a)}
        \includegraphics[width=0.9\textwidth]{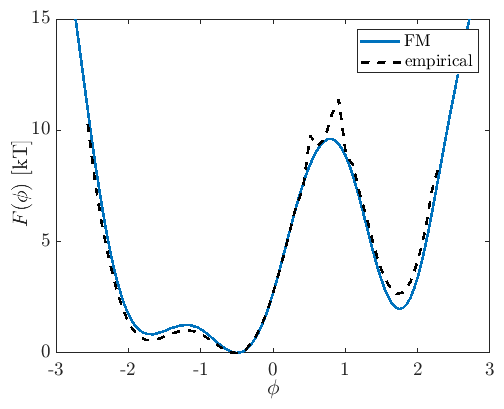}
    \end{minipage}
    \begin{minipage}[t]{0.49\textwidth}
        \centering
        \subfiguretitle{(b)}
        \includegraphics[width=0.9\textwidth]{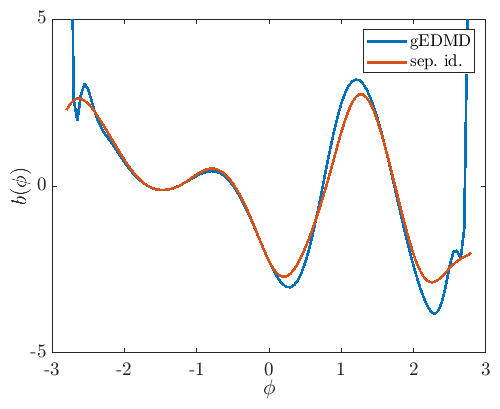}
    \end{minipage} \\[0.5ex]
    \begin{minipage}[t]{0.49\textwidth}
        \centering
        \subfiguretitle{(c)}
        \includegraphics[width=0.9\textwidth]{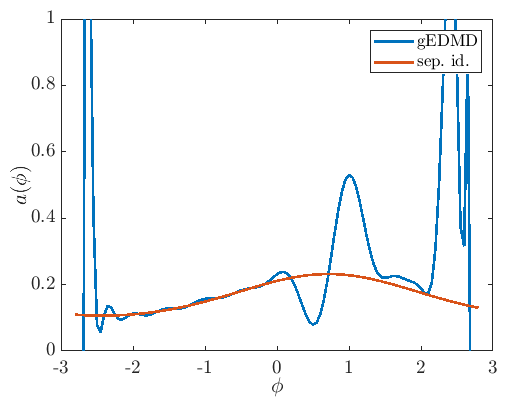}
    \end{minipage}
    \begin{minipage}[t]{0.49\textwidth}
        \centering
        \subfiguretitle{(d)}
        \vspace*{0.5ex}
        \includegraphics[width=0.9\textwidth]{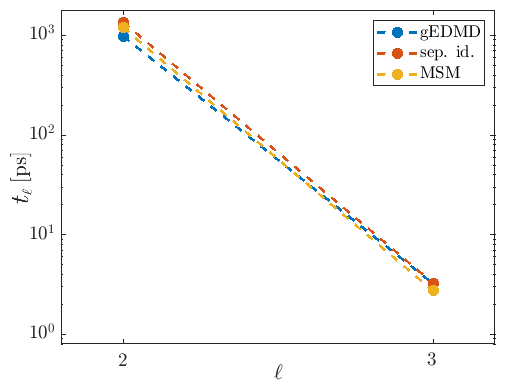}
    \end{minipage}
    \caption{Coarse-grained gEDMD along the $\phi$-angle coordinate of alanine dipeptide, using a basis set of 26 Legendre polynomials  (a) Effective energy $F^\phi$, as estimated by histogramming the molecular dynamics simulation data (black), and by applying force matching with a Gaussian basis set (blue). (b)~Estimates of the effective drift obtained directly from the generator matrix (blue), and by combining the solutions of \eqref{eq:minimization_problem_diffusion} and \eqref{eq:force_matching_problem} via \eqref{eq:rev_drift_diffusion_potential} (red). (c)~Estimates of the effective diffusion obtained directly from the generator matrix (blue), and by learning the diffusion via \eqref{eq:minimization_problem_diffusion} using a Gaussian basis set (red). (d)~Estimates of the two slowest implied time scales using a Markov state model (yellow), diagonalization of the generator matrix (blue), and diagonalization of the generator matrix corresponding to the optimal diffusion (red).}
    \label{fig:ala2}
\end{figure}

\subsection{Control}
\label{subsec:Control}

The predictive capabilities of the Koopman operator have also raised interest in the control community, where the aim is to determine a system input $u$ such that the non-autonomous control system $\dot{x} = b(x,u)$ behaves in a desired way, which results in the following control problem:
\begin{equation} \label{eq:OCP}
    \begin{aligned}
        \min_{u \in L^2([t_0,t_e], \R)} J(x,u) &= \min_{u \in L^2([t_0,t_e], \R)} \int_{t_0}^{t_e} \big\| x(t) - x^{\mathsf{ref}}(t) \big\|_2^2 + \alpha \|u(t)\|_2^2\ \textrm{d}t \\
        \mbox{s.t.}\qquad\qquad \dot{x}(t) &= b(x(t),u(t)), \\
        x(t_0) &= x_0.
    \end{aligned}
\end{equation}
In this formulation, the goal is to track a desired state over the control horizon $[t_0, t_e]$, and $\alpha \in \R^{>0}$ is a small number penalizing the control cost. In order to achieve a feedback behavior, problem \eqref{eq:OCP} is embedded into a model predictive control (MPC) \cite{GP17} scheme, where it has to be solved repeatedly over a relatively short horizon while the system (the plant) is running at the same time. The first part $[t_0, t_0 + h]$ of the optimal control $u$ is then applied to the plant, and \eqref{eq:OCP} has to be solved again on a shifted horizon $[t_0 + h, t_e + h]$.

Since the real-time requirements in MPC are often very hard to satisfy, a promising approach is to replace the system dynamics by a surrogate model, and one possibility is to use the Koopman operator or its generator for prediction. Introducing the variable $z=\psi(f(x))$, we obtain a linear system via the approximation $\mathbf{L}$ of the generator:
\begin{equation*}
    \dot{z}(t) \approx \mathbf{L} z(t).
\end{equation*}
However, as we see above, the Koopman operator is only defined for autonomous systems. Hence, a transformation has to be used (the exception being control-affine systems, where only the autonomous part needs to be modeled). In \cite{PBK15}, the control system was autonomized by introducing an augmented state $\widehat{x} = (x,u)^\top$, and DMD was performed on the augmented system. The same approach was also used in combination with MPC in \cite{KM18a}. This state augmentation significantly increases the data requirements (all combinations of states and control inputs should be covered), such that an alternative transformation was proposed in \cite{PK19,PK18} by restricting $u(t)$ to a finite set of inputs $\{u^1, \ldots, u^{n_c}\}$. This way, the control system can be replaced by a finite set of autonomous systems $b_{u^i}(x) = b(x,u^i)$ for which the corresponding generators $\{L_{u^1},\ldots,L_{u^{n_c}}\}$ can be approximated. The control task is thus to determine the optimal right-hand side in each time step instead of computing a continuous input $u$:
\begin{equation}\label{eq:OCP_Gen}
    \begin{aligned}
        \min_{u \in L^2([t_0,t_e], \{u^1, \ldots, u^{n_c}\})} &\int_{t_0}^{t_e} \big\|z(t) - z^{\mathsf{ref}}(t) \big\|_2^2 + \alpha \|u(t)\|_2^2\ \textrm{d}t \\
        \mbox{s.t.}\qquad\qquad \dot{z}(t) &= \mathbf{L}_{u(t)} z(t), \\
        z(t_0) &= \psi(f(x_0)).
    \end{aligned}
\end{equation}
Note that the quantization (i.e., the switching control) is encoded in the function space the control $u$ lives in. For a more detailed description, the reader is referred to \cite{PK19}.

Regardless of the approach, a drawback of Koopman operator based surrogate models is that the control freedom is limited by the finite lag time. While larger lag times are often beneficial for the approximation of the dynamics, this is counterproductive for control, as the control frequency is strongly limited. This issue is overcome by the generator approach~\eqref{eq:OCP_Gen} since we can choose arbitrary time steps here, and results on mixed integer optimal control problems (see, e.g., \cite{SBD12}) suggest that fast switches allow for solutions of any desired accuracy. Moreover, the continuous-time generator model is much better suited for switching time optimization approaches. Therein, the combinatorial problem of selecting the optimal right-hand side is replaced by a continuous optimization problem for the time instances $\tau_j$ at which the right-hand side is switched from the input $u^i$ to $u^{i+1}$:
\begin{equation}\label{eq:STO_Gen}
	\begin{aligned}
		\min_{\tau \in \R^p} &\int_{t_0}^{t_e} \big\|z(t) - z^{\mathsf{ref}}(t) \big\|_2^2 + \alpha \|u(t)\|_2^2\ \textrm{d}t \\
		\mbox{s.t.}\qquad\qquad \dot{z}(t) &= \mathbf{L}_{u^i} z(t), \qquad \mbox{for}~t\in[\tau_{j-1}, \tau_{j}), \\
		t_0 &= \tau_0 \leq \tau_1 \leq \ldots \leq \tau_p \leq t_e, \\
		i &= 1 + j~\mbox{mod}~n_c, \\
		z(t_0) &= \psi(f(x_0)).
	\end{aligned}
\end{equation}
By fixing the number $p$ of switches, this reformulation of \eqref{eq:OCP_Gen} is now a continuous, finite-dimensional optimization problem for the switching times with a given switching sequence (cf.~\cite{EWD03} for details), and both open and closed-loop control schemes (using MPC in the latter case) can be constructed.

Problem~\eqref{eq:STO_Gen} was also used in combination with the Koopman operator in \cite{PK19}, but the discrete-time system prohibits arbitrary switching points which results in a reduced performance. Using the generator solves this problem and additionally, there even exist efficient second-order methods for this problem class \cite{SOBG17}.

In what follows, we present examples for the two extensions for MPC based on the Koopman generator. For the deterministic case, we use the 1D viscous Burgers equation with varying lag times in an MPC framework (Problem~\eqref{eq:OCP_Gen}) and for the non-deterministic case, we control the expected value of an Ornstein--Uhlenbeck process using both MPC (Problem~\eqref{eq:OCP_Gen}) and open loop switching time control (Problem~\eqref{eq:STO_Gen}).

\subsubsection{Partial differential equations}

\begin{figure}
    \centering
    \begin{minipage}{0.49\textwidth}
        \centering
        \subfiguretitle{(a)}
        \includegraphics[width=0.85\textwidth]{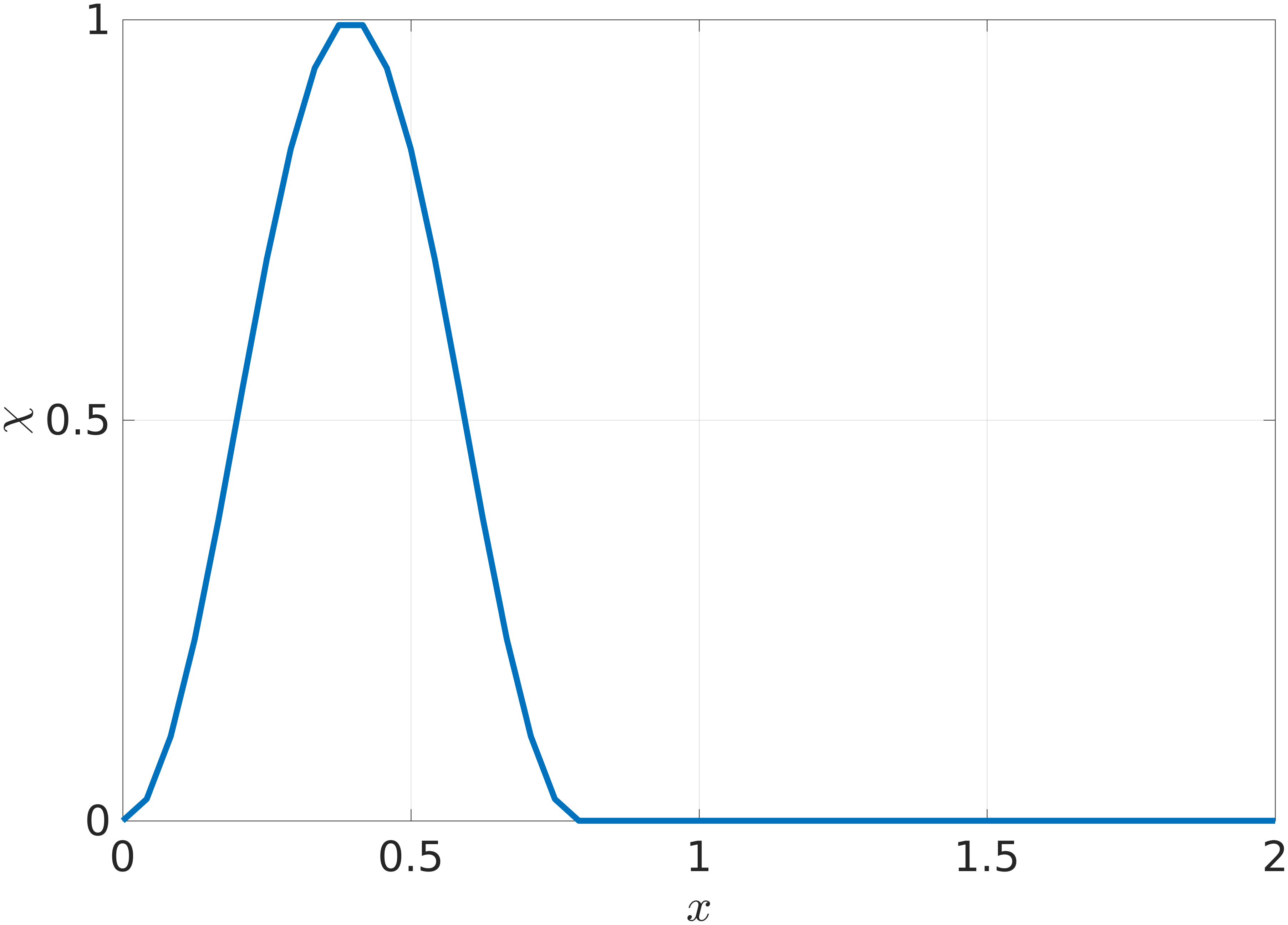}
    \end{minipage}
    \begin{minipage}{0.49\textwidth}
        \centering
        \subfiguretitle{(b)}
        \includegraphics[width=0.9\textwidth]{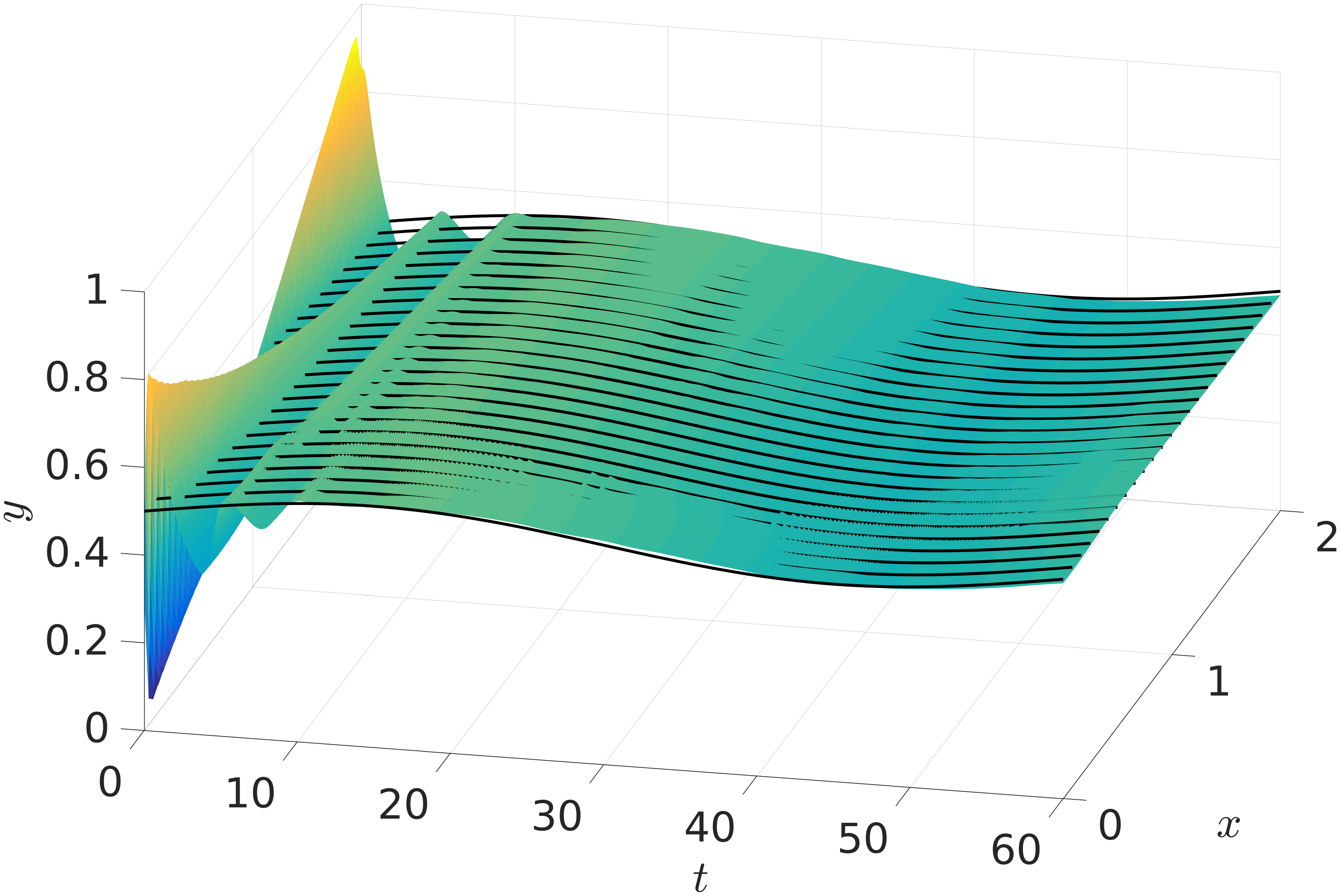}
    \end{minipage}
    \\[1ex]
    \begin{minipage}[t]{0.49\textwidth}
        \centering
        \subfiguretitle{(c)}
        \includegraphics[width=\textwidth]{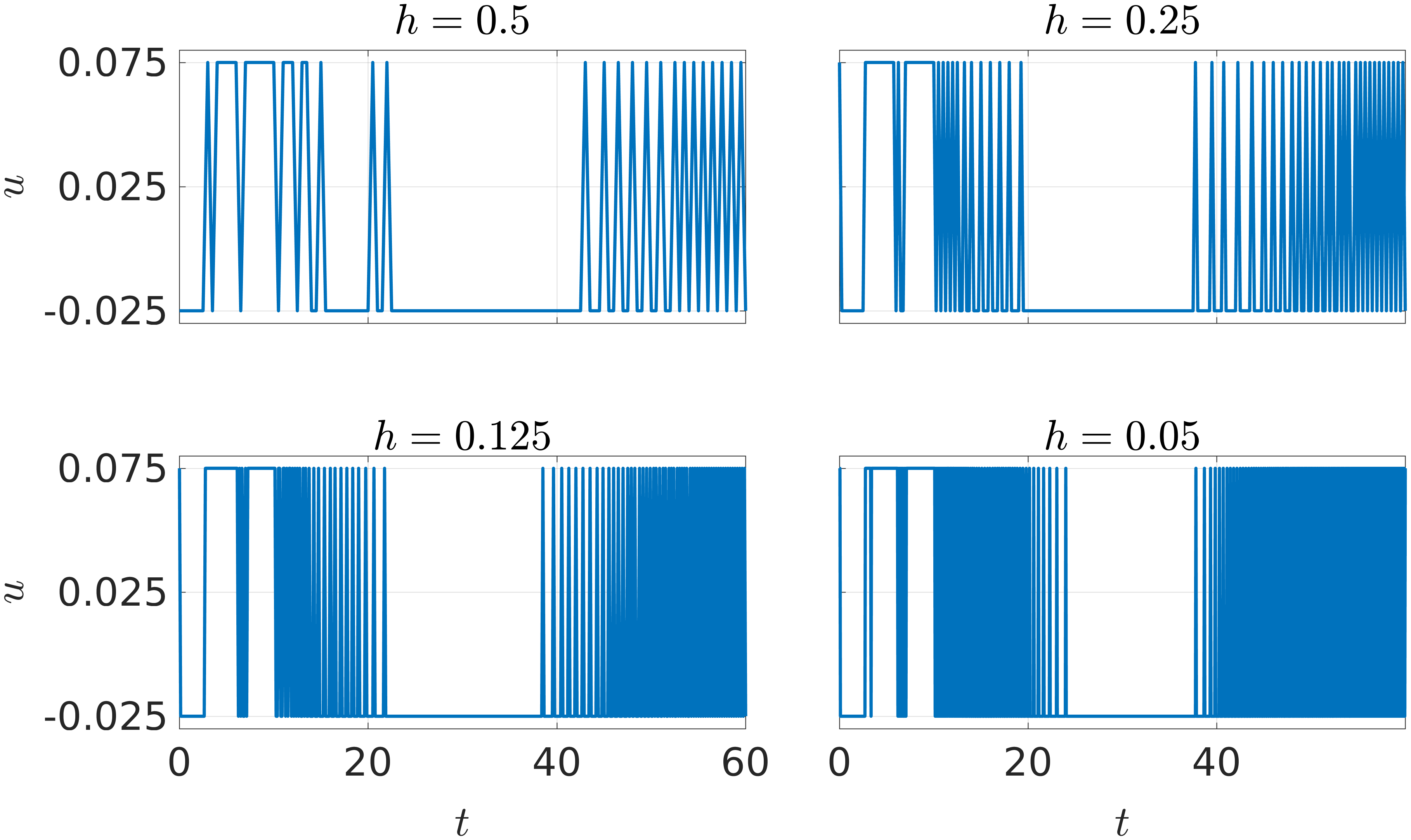}
    \end{minipage}
    \begin{minipage}[t]{0.49\textwidth}
        \centering
        \subfiguretitle{(d)}
        \includegraphics[width=0.85\textwidth]{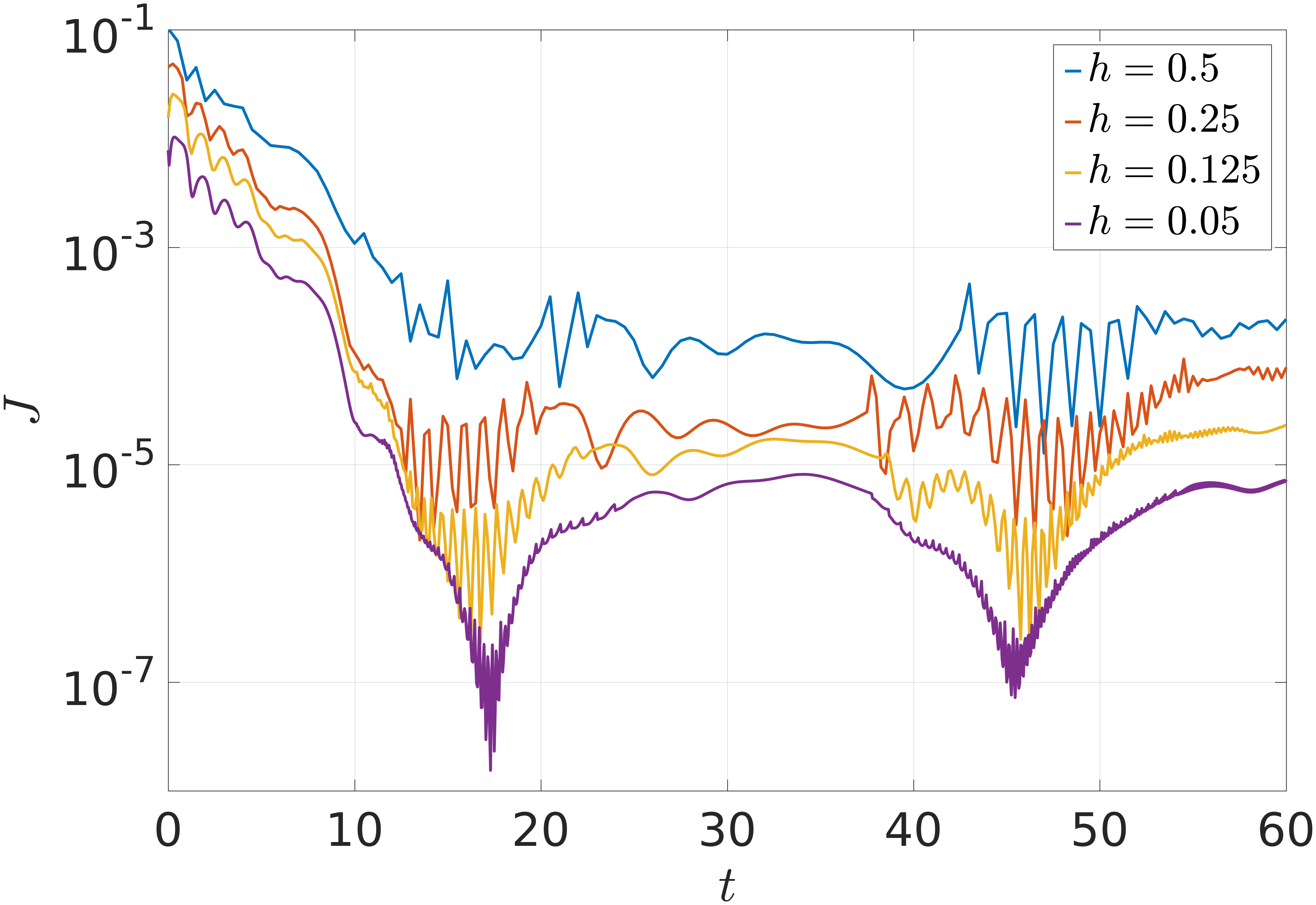}
    \end{minipage}
    \caption{Control of the Burgers equation using the Koopman generator and switching control. (a) The shape function used for the distributed control term. (b) The optimal state (colored) and the reference trajectories (black) for $h = 0.05$. (c) The optimal switching sequence as a function of the time step $h$. (d) The tracking error as a function of the time step $h$. }
    \label{fig:Burgers_control}
\end{figure}

Consider the 1D Burgers equation with distributed control
\begin{align*}
\dot{y}(t,x) - \nu \Delta y(t,x) + y(t,x) \nabla y(t,x) &= u(t) \chi(x).
\end{align*}
Here, $y$ denotes the state depending on space $x$ and time $t$, and the system is controlled by a shape function $\chi$ (see Figure~\ref{fig:Burgers_control} (a)) that can be scaled by the input $u\in\{-0.025, 0.075\}$. The objective is to track a sinusoidal reference trajectory (shown in black in (b)), and we do this by solving problem \eqref{eq:OCP_Gen} in an MPC framework. To this end, we approximate the Koopman generator using a relatively coarse ``full state observable'' (a grid of 25 equidistantly distributed points in space) and monomials up to order two. 
The data is collected from one trajectory with a piecewise constant input signal $u(t) \in \{ -0.025, 0.075\}$. It is then divided into two data sets corresponding to the constant inputs $0.025$ and $-0.075$, respectively. The time derivative $\dot{y}$ is computed via finite differences.

We see in (c) and (d) that with decreasing time steps $h$ (over which the input $u$ is constant) that the control performance increases significantly. While the time step $h=0.5$ corresponds to a solution that can be obtained by a Koopman operator approximation as well, the generator framework allows us to decrease the time steps and thereby the error by two orders of magnitude. 
Note that we can formally also decrease the lag time for the Koopman operator to increase the performance. In our experiments, the results were of comparable quality, and this is likely due to the high robustness of the MPC algorithm, which can cope well with small model inaccuracies. However, an advantage of the generator approach is that we can choose the time step adaptively---in contrast to the Koopman operator approach, where a change in the lag time requires a different data set and new computations. This can be beneficial in terms of computational efficiency and is thus particularly important for long control horizons (see, e.g., \cite{KWBS12}), due to which real-time capability may otherwise be jeopardized.

\subsubsection{Stochastic differential equations}

In the case of non-deterministic systems, the generator approach allows for a very elegant solution of \emph{stochastic control problems}. In stochastic (or \emph{robust}) control (see \cite{BM07,Mes16} for introductions), the goal is very often to steer the expected value to some desired value. In many situations, determining this expected value (e.g., via Monte Carlo methods) is numerically challenging. As the Koopman generator for stochastic systems describes the evolution of the expected value, see~\eqref{eq:stochastic_koopman_op}, problem \eqref{eq:OCP_Gen} can be used to solve a control problem for the expected value using a deterministic linear system. To this end, we replace the computation of the initial value by an average over the recent past: \[z_0 = \frac{1}{h} \int_{t_0-h}^{t_0} z(t) \ts \textrm{d}t.\]

We again consider the Ornstein--Uhlenbeck process from Example~\ref{ex:Ornstein Uhlenbeck}, with the only difference that we now add a control input:
\begin{equation*}
    \mathrm{d}X_t = -\alpha \ts (X_t - u) \ts \mathrm{d}t + \sqrt{2 \beta^{-1}} \ts \mathrm{d}W_t,
\end{equation*}
with $\alpha = 1$ and $\beta = 2$.
We compute two generator approximations corresponding to $u=-5$ and $u=5$ using monomials up to order 12. Figure~\ref{fig:OU_control}(a) shows the trajectories of the two systems and the predictions using the corresponding generators, and we see that the expected values are accurately predicted. We set $h = 0.05$ as a discretization for the control $u$ as well as the length of the input that is applied to the plant in each loop. The MPC controller based on~\eqref{eq:OCP_Gen} with the modified initial condition $z_0$ yields very good performance, as is shown for a tracking problem with a piecewise constant reference value in Figure~\ref{fig:OU_control}(b). The corresponding optimal control is shown in Figure~\ref{fig:OU_control}(c), and Figure~\ref{fig:OU_control}(d) shows that continuously varying inputs can be approximated equally well. 

Finally, we use the switching time reformulation \eqref{eq:STO_Gen} in an open loop fashion in order to track a $\tanh$ profile over 20 seconds. The results are shown in Figures~\ref{fig:OU_control}(e) and \ref{fig:OU_control}(f), where the optimal input with $p = 200$ switches is shown in (e) and the corresponding dynamics are shown in (f). We observe a remarkable performance, as the optimal trajectory of the generator model and the expected value of the controlled Ornstein--Uhlenbeck process (computed as the mean over 1000 simulations) are almost indistinguishable.

To summarize, the generator approach yields highly efficient control schemes both for open and closed-loop control, as the linear system for the prediction of the expected value requires no further sampling of multiple noisy trajectories.

\begin{figure}[h!]
    \centering
    \begin{minipage}[t]{0.49\textwidth}
        \centering
        \subfiguretitle{(a)}
        \vspace*{0.5ex}
        \includegraphics[width=0.9\textwidth]{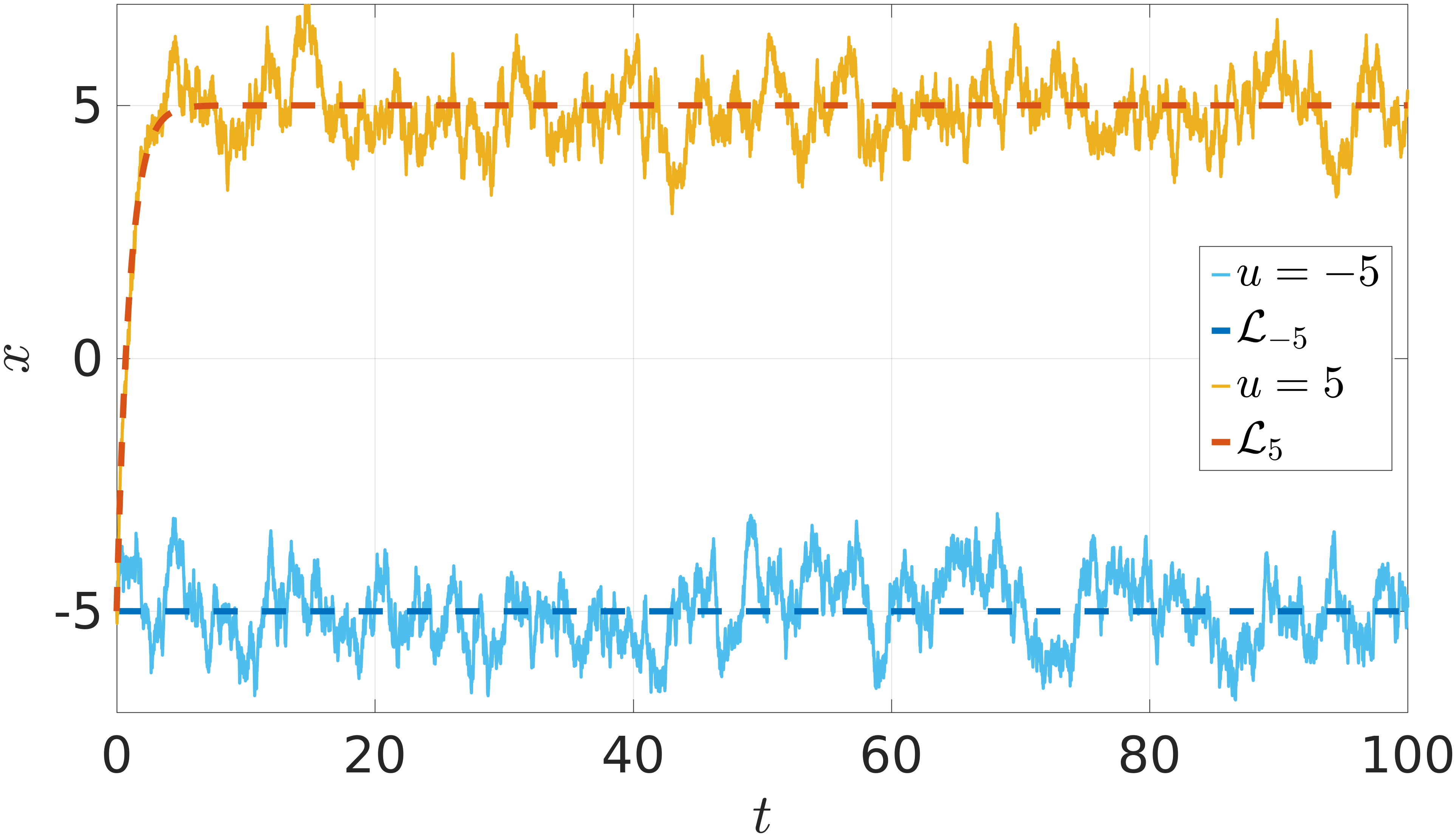}
    \end{minipage}
    \begin{minipage}[t]{0.49\textwidth}
        \centering
        \subfiguretitle{(b)}
        \includegraphics[width=0.9\textwidth]{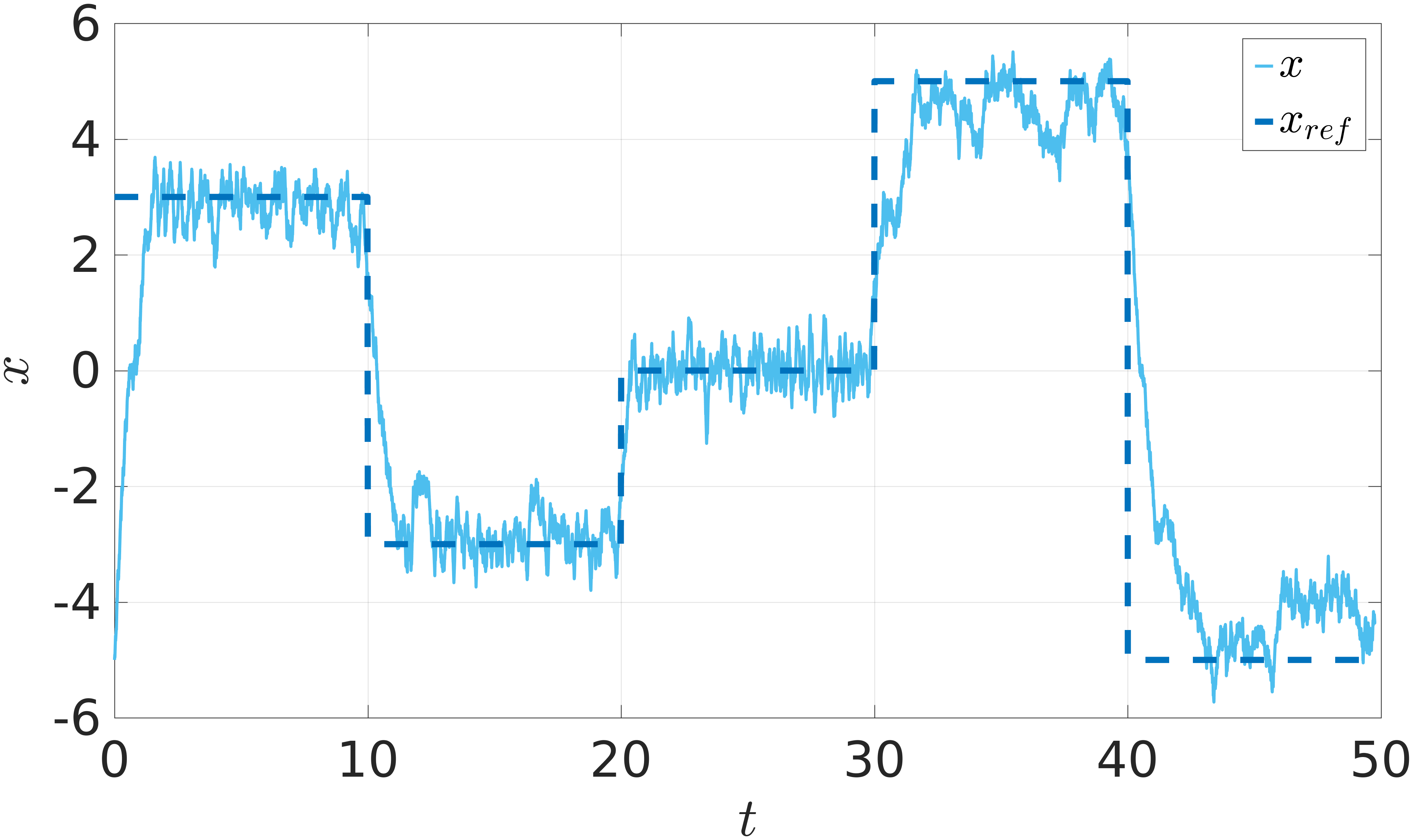}
    \end{minipage}
    \\[1ex]
    \begin{minipage}[t]{0.49\textwidth}
        \centering
        \subfiguretitle{(c)}
        \vspace*{0.3ex}
        \includegraphics[width=0.9\textwidth]{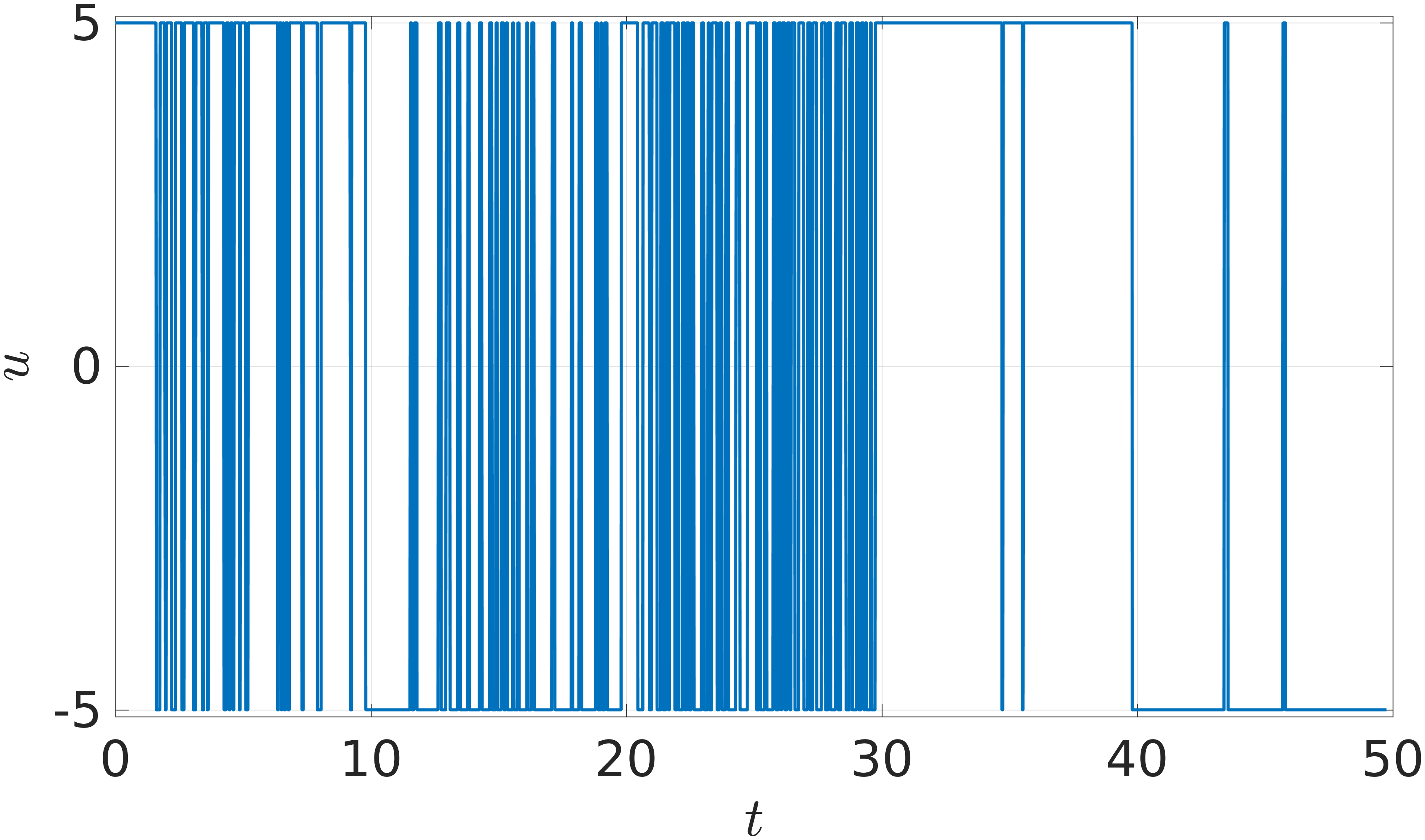}
    \end{minipage}
    \begin{minipage}[t]{0.49\textwidth}
        \centering
        \subfiguretitle{(d)}
        \vspace*{0.5ex}
        \includegraphics[width=0.9\textwidth]{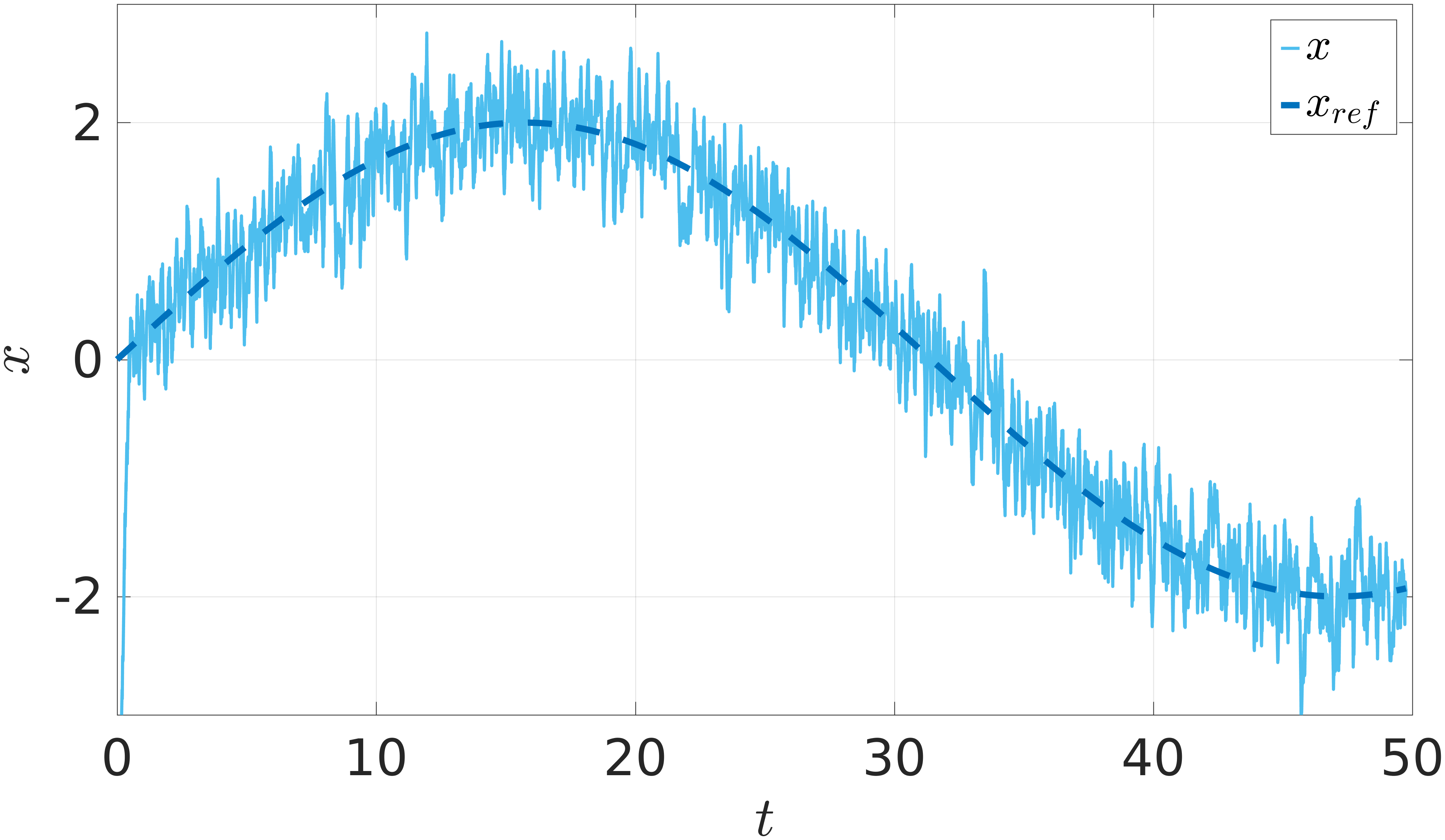}
    \end{minipage}
	\\[1ex]
	\begin{minipage}[t]{0.49\textwidth}
		\centering
		\subfiguretitle{(e)}
		\vspace*{0.3ex}
		\includegraphics[width=0.9\textwidth]{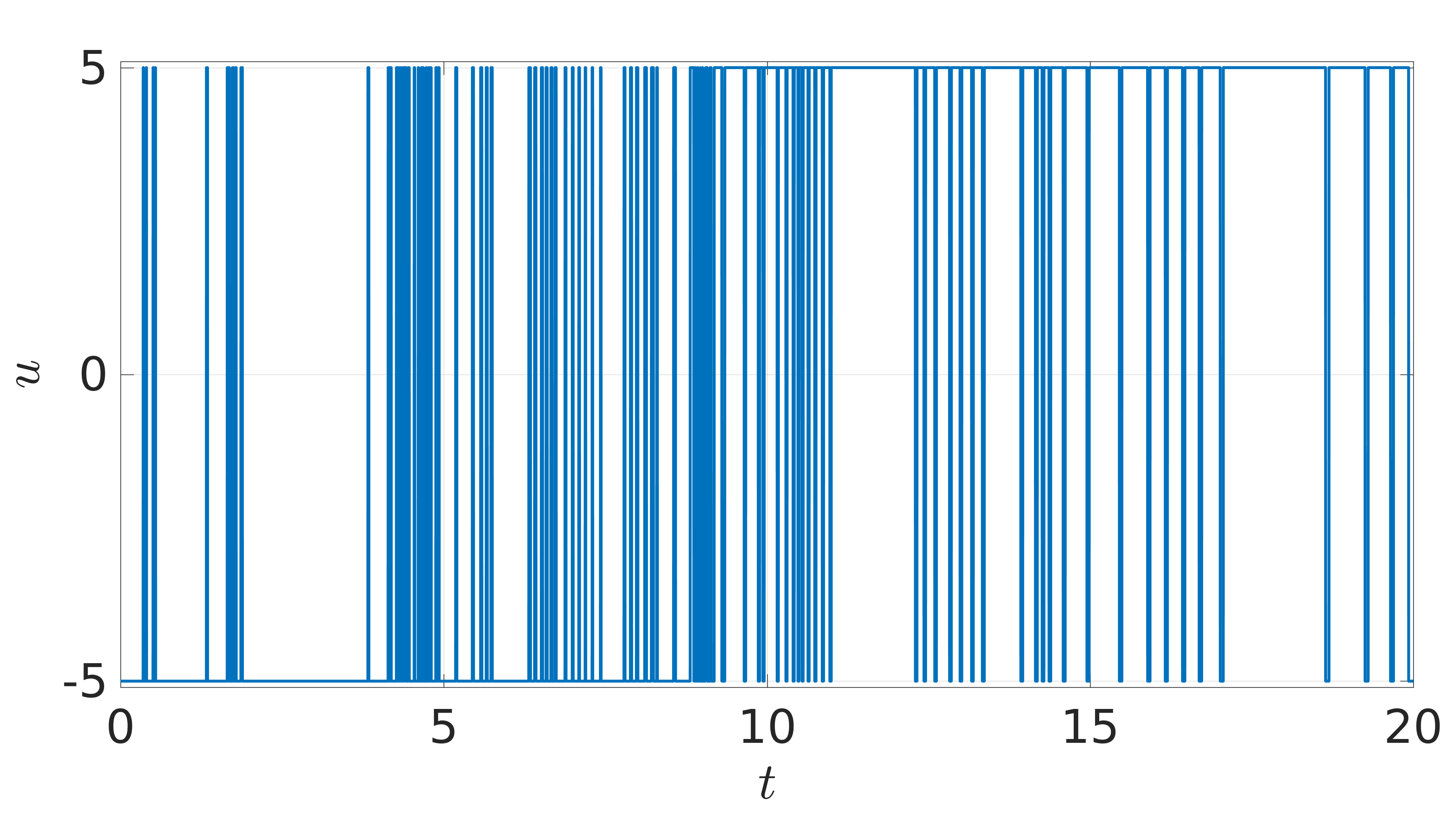}
	\end{minipage}
	\begin{minipage}[t]{0.49\textwidth}
		\centering
		\subfiguretitle{(f)}
		\vspace*{0.5ex}
		\includegraphics[width=0.9\textwidth]{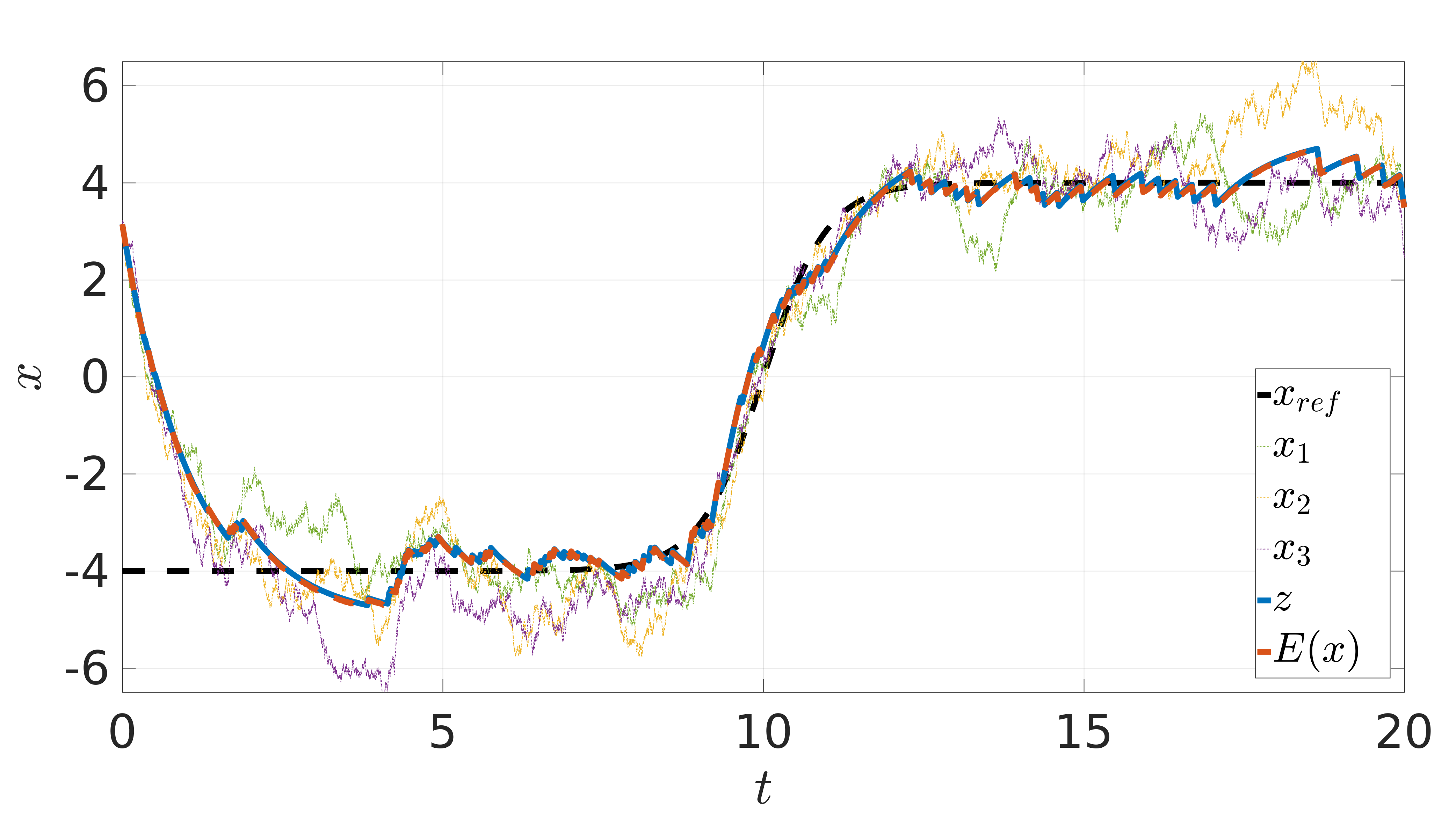}
	\end{minipage}
    \caption{Control of the expected value of the Ornstein--Uhlenbeck process. (a) Simulation and generator prediction for $u=5$ and $u=-5$, respectively. (b) Tracking of a piecewise constant reference trajectory using MPC based on Problem~\eqref{eq:OCP_Gen}. (c) The corresponding optimal input signal. (d) MPC based tracking of a continuous reference trajectory. (e)~Solution of Problem \eqref{eq:STO_Gen} with $p=200$ for the reference trajectory $x_{ref} = \mbox{tanh}(t-10)$, see (f). (f) The corresponding optimal trajectories of the generator model ($z$, blue line), of three realizations of the Ornstein--Uhlenbeck process ($x_1$ to $x_3$, dotted lines), and of the expected value of the controlled process ($E(x)$, dashed orange line).}
    \label{fig:OU_control}
\end{figure}

\section{Conclusion}
\label{sec:Conclusion}

We presented an extension of standard EDMD to approximate the generator of the Koopman or Perron--Frobenius operator from data and highlighted several important applications pertaining to model reduction, system identification, and control. We illustrated that this approach can be used to obtain a decomposition into eigenvalues, eigenfunctions, and modes and, furthermore, that SINDy emerges as a special case. The proposed methods were implemented in Python, the gEDMD code and some of the above examples are available at \url{https://github.com/sklus/d3s/}.

Open questions include the convergence of gEDMD if not only the number of data points but also the number of basis functions tends to infinity. It is also unclear which part of the spectrum is approximated if the generator does not possess a pure point spectrum. Furthermore, is it possible to learn coarse-grained dynamics by only considering the dominant terms of the decomposition of the system's equations into eigenvalues, eigenfunctions, and modes (cf.\ Example~\ref{ex:systemidentification} and also \cite{NBC19})? Another interesting application of gEDMD would be to compute committor functions or hitting times. Extensions to non-autonomous systems will be considered in future work.

\section*{Acknowledgements}

S.~K., J.~N., and C.~S were funded by Deutsche Forschungsgemeinschaft (DFG) through grant CRC 1114 (Scaling Cascades in Complex Systems, project ID: 235221301) and through Germany's Excellence Strategy (MATH\texttt{+}: The Berlin Mathematics Research Center, EXC-2046/1, project ID: 390685689). F.~N.\ was partially funded by the Rice University Academy of Fellows. F.~N.\ and C.~C.\ were supported by the National Science Foundation (CHE-1265929, CHE-1738990, CHE-1900374, PHY-1427654) and the Welch Foundation (C-1570). C.~C.\ also acknowledges funding from the Einstein Foundation Berlin. S.~P.\ acknowledges support by the DFG Priority Programme 1962.

{\small{}\bibliographystyle{unsrturl}
\bibliography{gEDMD}
}{\small\par}

\end{document}